\documentclass[11pt,twoside,reqno]{amsart}
\usepackage[dvips,bookmarks=false]{hyperref}
\usepackage{amsmath,amssymb,verbatim}
\usepackage[all]{xy}
\usepackage{graphicx}
\usepackage{geometry}
\date{\today}

\def\R{{\mathbf R}}
\def\C{{\mathbf C}}

\def\Q{{\mathbf Q}}

\def\1{\mathbf 1}

\def\Z{{\mathbf Z}}
\def\N{{\mathbf N}}

\def\cS{{\mathcal S}}

\def\I{{\mathcal I}}

\def\sf{{h}}
\renewcommand{\=}{:=}

\newcommand{\tred}{T_\mathrm{red}}

\def\be{\begin{equation}}
\def\ee{\end{equation}}

\DeclareMathOperator{\spann}{span}
\DeclareMathOperator{\interior}{Int}
\DeclareMathOperator{\pic}{Pic}
\DeclareMathOperator{\cl}{Cl}
\DeclareMathOperator{\id}{Id}
\DeclareMathOperator{\stjarna}{Star}

\newtheorem{thm}{Theorem}[section]
\newtheorem{lma}[thm]{Lemma}
\newtheorem{cor}[thm]{Corollary}
\newtheorem{prop}[thm]{Proposition}

\newtheorem*{thmA}{Theorem A}
\newtheorem*{thmA'}{Theorem A'}
\newtheorem*{thmB'}{Theorem B'}
\newtheorem*{thmB}{Theorem B}
\newtheorem*{thmC}{Theorem C}
\newtheorem*{thmC'}{Theorem C'}
\theoremstyle{definition}

\theoremstyle{remark}

\newtheorem{preremark}[thm]{Remark}
\newtheorem{preex}[thm]{Example}

\newenvironment{remark}{\begin{preremark}}{\qed\end{preremark}}
\newenvironment{ex}{\begin{preex}}{\qed\end{preex}}

\numberwithin{equation}{section}

\begin{document}

\title[Stabilization]{Stabilization of monomial maps}

\date{\today}
\thanks{First author partially supported by the NSF.\ Second author partially supported by the Swedish Research Council and the NSF}
\author{Mattias Jonsson}
\author{Elizabeth Wulcan}

\address{Dept of Mathematics, University of Michigan, Ann Arbor \\ MI 48109-1043\\ USA}

\email{mattiasj@umich.edu, wulcan@umich.edu}



\begin{abstract}
  A monomial (i.e.\ equivariant) selfmap 
  of a toric variety is called stable
  if its action on the Picard group commutes with iteration.
  Generalizing work of Favre to higher dimensions,
  we show that under suitable conditions,
  a monomial map
  can be made stable by refining the underlying fan.
  In general, the resulting toric variety has quotient singularities;
  in dimension two we give criteria for when it can be chosen
  smooth, as well as examples when it cannot.
\end{abstract}

\maketitle

\section*{Introduction}
An important part of higher-dimensional complex dynamics concerns the 
construction of currents and measures that are
invariant under a given meromorphic selfmap $f:X\dashrightarrow X$ of
a compact complex manifold $X$. 
In doing so, it is often desirable that the 
action of $f$ on the cohomology of $X$ be compatible with iteration; 
following Sibony~\cite{Sibony} (see also~\cite{FS2}) we then call $f$ 
(algebraically) \emph{stable}. 

If $f$ is not stable, we can try to make a bimeromorphic change of 
coordinates $X'\to X$ such that the induced selfmap 
of $X'$ becomes stable.
Understanding when this is possible seems 
to be a difficult problem. 
On the one hand, Favre~\cite{Favre} showed that 
stability is not always achievable. 
On the other hand, it can be achieved for 
bimeromorphic
maps of surfaces~\cite{DF}, for a large class of monomial mappings 
in dimension two~\cite{Favre} (more on this below) 
and for polynomial maps of $\C^2$~\cite{dyncomp}. 
Beyond, these classes, very little seems to be known. 

In this paper, we study the stabilization problem for 
monomial (or equivariant) maps of
toric varieties, extending certain results of Favre to higher dimensions.
A toric variety $X=X(\Delta)$ is defined by a lattice $N\cong\Z^m$ 
and fan $\Delta$ in $N$. A monomial selfmap $f:X\dashrightarrow X$
corresponds to a $\Z$-linear map $\phi:N\to N$.
See Sections~\ref{sec:toric} and~\ref{sec:monmaps} for more details.

We work in codimension one
and say that $f$ is \emph{1-stable} if $(f^n)^*=(f^*)^n$,
where $f^*$ denotes the action on the Picard group of $X$.
Geometrically, this means that no iterate of $f$ sends a
hypersurface into the indeterminacy set of $f$~\cite{FS2,Sibony}.
\begin{thmA}
  Let $\Delta$ be a fan in a lattice $N\cong\Z^m$, 
  and $f:X(\Delta)\dashrightarrow X(\Delta)$ a monomial map. 
  Assume that the eigenvalues of the associated linear map
  $\phi:N_\R\to N_\R$ are real and satisfy 
  $\mu_1 >\mu_2 > \dots >\mu_m>0$. 
  Then there exists a complete 
  simplicial refinement $\Delta'$ of $\Delta$ 
  such that $X(\Delta')$ is projective and the induced map 
  $f:X(\Delta')\dashrightarrow X(\Delta')$ is $1$-stable.
\end{thmA}
Here $N_\R$ denotes the vector space $N\otimes_\Z\R$.
The variety $X'=X(\Delta')$ will not be smooth in general
but it will have at worst quotient singularities. 
We can pick $X'$ smooth at the expense of replacing
$f$ by an iterate (but allowing more general $\phi$):
\begin{thmA'}
  Let $\Delta$ be a fan in a lattice $N$ of rank $m$, 
  and let $f:X(\Delta)\dashrightarrow X(\Delta)$ be a monomial map. 
  Suppose that the eigenvalues of the 
  associated linear map $\phi:N_\R\to N_\R$
  are real and satisfy  $|\mu_1| >|\mu_2| >\dots >|\mu_m|>0$. 
  Then there exists a complete (regular) refinement $\Delta'$ of 
  $\Delta$ and $n_0\in\N$, such that $X(\Delta')$ 
  is smooth and projective and the induced map
  $f^n:X(\Delta')\dashrightarrow X(\Delta')$ is 
  1-stable for all $n\geq n_0$. 
\end{thmA'}
When the fan we start with is trivial, that is,
the initial toric variety is the torus $(\C^*)^m$,
we can relax the assumptions on the eigenvalues
slightly and obtain:
\begin{thmB}
  Let $f:(\C^*)^m\to(\C^*)^m$ be a monomial map. 
  Suppose that the associated linear map $\phi:N_\R\to N_\R$ is
  diagonalizable, with real eigenvalues 
  $\mu_1>\mu_2\geq\mu_3\geq\dots\geq\mu_m>0$. 
  Then there exists a complete simplicial fan $\Delta$ 
  such that $X(\Delta)$ is projective and 
  $f:X(\Delta)\dashrightarrow X(\Delta)$ is 1-stable.
\end{thmB}
It is unclear whether $X(\Delta)$ can be chosen smooth
in Theorem~B, even if we replace $f$ by an iterate, 
see Remark~\ref{R301}.
Picking $X(\Delta')$ smooth in Theorem~A (without passing to
an iterate) also seems quite delicate.
We address the latter problem only in dimension two:
\begin{thmC}
  Let $\Delta$ be a fan in a lattice $N$ of rank $m=2$, 
  let $f:X(\Delta)\dashrightarrow X(\Delta)$ be a monomial map,
  and let $\mu_1$, $\mu_2$ be the eigenvalues of the 
  associated linear map $\phi:N_\R\to N_\R$, labeled so that
  $|\mu_1|\ge|\mu_2|$. 
  Suppose that any of the following conditions holds:
  \begin{enumerate}
  \item[(a)]
    $|\mu_2|<1$;
  \item[(b)]
    $|\mu_1|>|\mu_2|$ and $\mu_1, \mu_2\in \Z$;
  \item[(c)]
    $\mu_1,\mu_2\in\R$ and $\mu_1 \mu_2>0$.
  \end{enumerate}
  Then there is a complete (regular) refinement $\Delta'$ of $\Delta$ 
  such that $X(\Delta')$ is smooth (and projective) and 
  $f:X(\Delta')\dashrightarrow X(\Delta')$ is $1$-stable.
\end{thmC}
Example~\ref{motex1} shows that Theorem~C may fail 
when $|\mu_1|>|\mu_2|>1$, whereas 
Example~\ref{motex2} and~\cite[Exemple~2]{Favre}
show that it may fail 
when $|\mu_1|=|\mu_2|$ and $\mu_1/\mu_2$ is 
root of unity different from $1$.

Theorem~C should be compared with the work of Favre~\cite{Favre},
in which the following result is proved.
\begin{thmC'}
  Let $\Delta$ be a fan in a lattice $N$ of rank $m=2$, 
  let $f:X(\Delta)\dashrightarrow X(\Delta)$ be a monomial map
  and let $\mu_1$, $\mu_2$ be the eigenvalues of the 
  associated linear map $\phi:N_\R\to N_\R$.
  Then we are in precisely one of the following cases:
  \begin{itemize}
  \item[(i)]
    $\mu_1$, $\mu_2$  are complex conjugate and 
    $\mu_1/\mu_2$ is not a root of unity;
  \item[(ii)]
    there exists a complete refinement $\Delta'$ of $\Delta$ 
    such that the induced map 
    $f:X(\Delta')\dashrightarrow X(\Delta')$ is $1$-stable.
  \end{itemize}
\end{thmC'}
Here $X(\Delta')$ is not necessarily smooth.
The main result in~\cite{Favre} also asserts that we can
make $f$ 1-stable on a smooth toric variety by allowing
ramified covers, but there is a gap in this argument,
see Remark~\ref{R201}.

While monomial maps are quite special, 
they are interesting in their own right.
We refer to the article by Hasselblatt and Propp~\cite{propp}
for more information, and to the paper by Bedford and Kim~\cite{BK2}
for the problem---related to stability---of characterizing 
monomial maps whose degree growth sequence satisfies a linear
recursion formula.
For nonmonomial maps in higher dimensions, 
stability or degree growth is only understood 
in special cases~\cite{BK1,Nguyen}.

We note that many of the results in this paper have 
been obtained independently by Jan-Li Lin~\cite{Lin}. 
In particular, \cite[Thm~5.7~(a)]{Lin} 
coincides with our Theorem B' in Section~\ref{sec:proofB}.

There is a conjectured relationship between the eigenvalues $\mu_j$ and 
the \emph{dynamical degrees} $\lambda_j$, $1\le j\le m$ of $f$ 
(see~\cite{RuSh,DS1,Gu2} for a definition of dynamical degrees). 
Namely, the conjecture says that $\lambda_j=|\mu_1|\dots |\mu_j|$. 
See~\cite{FW} for work in this direction.
Granting this formula, the condition 
$|\mu_1|>\dots>|\mu_m|$ is equivalent to 
$j\mapsto \log \lambda_j$ being strictly concave. 
Now the conjecture does hold in dimension two.
This means that~(a) in 
Theorem~C is equivalent to (a')~$\lambda_2 <\lambda_1$ 
and~(b) is equivalent to~(b')~$\lambda_2<\lambda_1^2$ and $\lambda_1\in\Z$. 
Also observe that~(c) is satisfied for $f^2$ 
as soon as $\lambda_2<\lambda_1^2$. 

\medskip
To prove the theorems above, we translate them into statements about 
the linear map $\phi:N\to N$. 
What we ultimately prove is that we can refine the original 
fan $\Delta$ (by adding cones) so that the new fan $\Delta'$ contains
a finite collection $\mathcal{T}$ of invariant cones 
(i.e.\ $\phi$ maps each cone into itself) 
that together attract all one-dimensional cones in $\Delta'$.
More precisely, for every one-dimensional cone $\rho\in\Delta'$,
there exists $n_0\ge0$ such that $\phi^n(\rho)$ lies in a cone
in $\mathcal{T}$ for $n\ge n_0$ and $\phi^n(\rho)$ is a 
one-dimensional cone in $\Delta'$ for $0\le n<n_0$,
see Corollary~\ref{invkoner}.

Constructing such a collection of cones is also the strategy 
by Favre~\cite{Favre} for proving Theorem~C'.
In fact, the proof of Theorem~B is 
a straightforward adaptation of arguments in~\cite{Favre}. 
Indeed, the dynamics of $\phi:N_\R\to N_\R$ is easy to
understand: under iteration, a typical vector $v$ tends to move
towards the one-dimensional eigenspace associated to the 
largest eigenvalue $\mu_1$ of $\phi$. 
We can therefore find a simplicial cone $\sigma$ of maximum
dimension which is invariant under $\phi$; it will contain 
an eigenvector $e_1$ corresponding to $\mu_1$, in its interior. 
Using this cone we easily construct a fan for which Theorem~B holds.

On the other hand, Theorems~A and~A' are much more delicate
as we have to take into account the original fan $\Delta$.
For example, the simple argument for Theorem~B outlined above
will not work in general, as it is possible that the one-dimensional
cone $\R_+e_1$ is rational and belongs to $\Delta$. Moreover, 
there may be one-dimensional rays in $\Delta$ that are not attracted
to $\R_+e_1$ under iteration. Thus we must proceed more
systematically, and this is where the argument becomes 
significantly more involved in higher dimensions.

What we do is to look at the set $\tred(\phi)$
of all invariant rational subspaces 
$V\subseteq N_\R$. This means that $\phi(V)=V$ and that
$V\cap N_\Q$ is dense in $V$. It turns out that 
$\tred(\phi)$ is a finite set that admits a natural 
tree structure determined by the dynamics.
Using this tree we inductively 
construct a collection $\mathcal T$ of
invariant rational cones that together attract any lattice
point in $N$. The construction is flexible enough so that
the cones in $\mathcal T$ are ``well positioned'' with respect to 
the original fan $\Delta$.
In particular, each cone in $\mathcal T$ is contained in a unique cone
in $\Delta$. This allows us to refine $\Delta$ into a fan that 
contains all cones in $\mathcal T$. Significant care is called for,
however, as the construction is done inductively over
a tree, and incorporating a new cone into a given fan will 
require many cones of the original fan to be subdivided.
The actual construction is therefore more technical
than may be expected.

In dimension two, these difficulties are largely invisible.
They are the reason why we in Theorem~A impose stronger conditions
on the eigenvalues than Favre did in~\cite{Favre}. 
It would be interesting to try to weaken the conditions in Theorem~A.

\smallskip
The proof of Theorem~C is of a quite different nature, 
and uses the original
ideas of Favre as well as some methods from classical 
number theory. Indeed, some of the arguments are parallel to the 
analysis of the continued fractions expansion of quadratic 
surds~\cite{HardyWright}.

The paper is organized as follows. 
In Sections~\ref{sec:toric} and~\ref{sec:monmaps} we discuss 
toric varieties and monomial mappings. 
Section~\ref{sec:dimtwo} is concerned with the two-dimensional 
situation, namely the proofs of Theorems~C and~C', and examples
showing that smooth stabilization is not always possible.
Then in Sections~\ref{sec:proofA} and~\ref{sec:proofB}
we return to the higher-dimensional case and prove
Theorems~A,~A' and~B. Finally, in Section~\ref{sec:examples}
we illustrate our proof of Theorem~A in dimensions two and three
and give a counterexample to the statement in Theorem~A 
when the eigenvalues have mixed signs.

\textbf{Acknowledgment}
We thank Alexander Barvinok, Stephen DeBacker,
Charles Favre, Jan-Li Lin and Mircea~Musta\c{t}\u{a}
for fruitful discussions.
We also thank the referee for 
several helpful suggestions, in particular that
the toric varieties in Theorems~A and~A' could be made projective.

\section{Toric varieties}\label{sec:toric}
A toric variety is a (partial) compactification of the torus $T\cong (\C^*)^m$, which contains $T$ as a dense subset and which admits an action of $T$ that extends the natural action of $T$ on itself. We briefly recall some of the basic definitions, referring to~\cite{F1} and~\cite{Oda} for details.

\subsection{Fans and toric varieties}\label{fansoch}
Let $N$ be a lattice isomorphic to $\Z^m$ and let $M=\text{Hom}(N,\Z)$ denote the dual lattice. Set $N_\Q:=N\otimes_\Z \Q$, $N_\R:=N\otimes_\Z \R$, and $N_\C:=N\otimes_\Z \C$.

A \emph{cone} $\sigma$ in $N_\R$ is a set that is closed under positive scaling. If $\sigma$ is convex and does not contain any line in $N_\R$ it is said to be \emph{strictly convex}. If $\sigma$ is of the form $\sigma=\sum\R_+v_i$ for some $v_i\in N$, 
we say that $\sigma$ is a \emph{convex rational cone} \emph{generated} by the vectors $v_i$. A \emph{face} of $\sigma$ is the intersection of $\sigma$ and a supporting hyperplane. The \emph{dimension} of $\sigma$ is the dimension of the linear space $\R\cdot \sigma$ spanned by $\sigma$. One-dimensional cones are called \emph{rays}. Given a ray $\sigma$, the associated \emph{primitive vector} is the first lattice point met along $\sigma$. A $k$-dimensional cone is \emph{simplicial} if it can be generated by $k$ vectors. A cone is \emph{regular} if it is generated by part of a basis for $N$. If $\sigma$ is a rational cone we denote by $\interior\sigma$ the \emph{relative interior} of $\sigma$, that is,
$\interior\sigma$ are the elements that are in $\sigma$ but not in any proper face of $\sigma$. If $\sigma$ is generated by $v_i$, then $\interior\sigma=\sum \R_+^*v_i$. Write $\partial\sigma:=\sigma\setminus\interior\sigma$. 

A \emph{fan} $\Delta$ in $N$ is a finite collection of rational strongly convex cones in $N_\R$ such that each face of a cone in $\Delta$ is also a cone in $\Delta$, and moreover the intersection of two cones in $\Delta$ is a face of both of them. The last condition could also be replaced by the relative interiors of the cones in $\Delta$ being disjoint. Note that a fan is determined by its maximal cones with respect to inclusion. Let $\Delta(k)$ denote the set $k$-dimensional faces of $\Delta$. The \emph{support} $|\Delta|$ of the fan $\Delta$ is the union of all cones of $\Delta$. In fact, given any collection of cones $\Sigma$, we will use $|\Sigma|$ to denote the union of the cones in $\Sigma$. 
If $|\Delta|=N_\R$, then the fan $\Delta$ is said to be \emph{complete}.
If all cones in $\Delta$ are simplicial, then $\Delta$ is said to be \emph{simplicial}, and if all cones in $\Delta$ are regular, then $\Delta$ is said to be \emph{regular}. 
A \emph{sub-fan} of a fan $\Delta$ is a fan $\tilde\Delta$ with 
$\tilde\Delta\subseteq\Delta$.
A fan $\Delta'$ is a \emph{refinement} of $\Delta$ if each cone in $\Delta$ is a union of cones in $\Delta'$. 
Every fan admits a regular refinement.

A strongly convex rational cone $\sigma$ in $N$ determines an affine toric variety $U_\sigma$ and a fan $\Delta$ determines a toric variety $X(\Delta)$, obtained by gluing together the $U_\sigma$ for $\sigma\in\Delta$.
The variety $U_\tau$ is dense in $U_\sigma$ if $\tau$ is a face of $\sigma$. 
In particular, the torus $T_N:=U_{\{0\}}=N\otimes_\Z\C^*\cong (\C^*)^m$, 
is dense in $X(\Delta)$. The torus acts on $X(\Delta)$, the orbits 
being exactly the varieties $U_\sigma$, $\sigma\in\Delta$.

A toric variety $X(\Delta)$ is compact if and only if $\Delta$ is complete.
Toric varieties are normal and Cohen-Macaulay. 
If $\Delta$ is simplicial, 
then $X(\Delta)$ has at worst quotient singularities,
and  $X(\Delta)$ is non-singular if and only if $\Delta$ is regular.

\subsection{Incorporation of cones}
To prove Theorems A, A', and C we will refine fans by adding certain cones. 
The following lemma is probably well known; we learned it from A. Barvinok. The techniques in the proof will not be used elsewhere in the paper.

\begin{lma}\label{barvinok}
Let $\Delta$ be a simplicial fan and let $\sigma_0\in\Delta$. Assume that $\sigma_1\subseteq \sigma_0$ is a simplicial cone, such that $\partial \sigma_1\cap\partial\sigma_0$ is a face of both $\sigma_1$ and $\sigma_0$. Then there exists a simplicial refinement $\Delta'$ of $\Delta$ such that $\sigma_1\in\Delta'$, and if $\sigma\in\Delta$ satisfies $\sigma\not\supseteq\sigma_0$, then $\sigma\in\Delta'$. Moreover, all rays in $\Delta'(1)\setminus\Delta(1)$ are one-dimensional faces of $\sigma_1$. 
\end{lma}
For examples of cones $\sigma_1\subseteq\sigma_0$, see Figure 1. 

\begin{figure}\label{barv}
\begin{center}
\includegraphics{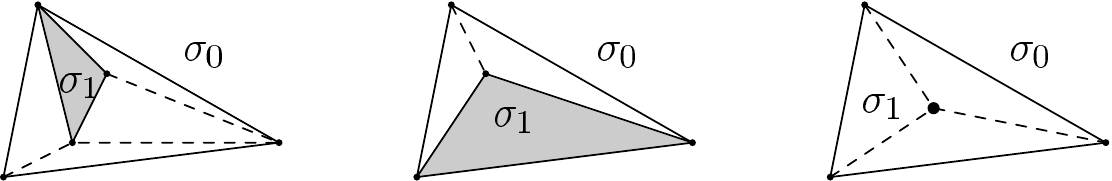}
\caption{Examples of cones $\sigma_1\subseteq \sigma_0$ in 
  Lemma~\ref{barvinok}. The cone $\sigma_0$ is three-dimensional; 
  the figure shows the intersection with an affine plane. 
  The dashed lines indicate the fan $\Delta_0$ in the proof of 
  Lemma~\ref{barvinok}.}
\end{center}
\end{figure}
\begin{proof}
Following~\cite[p.~129~ff]{Ziegler} we construct a fan $\Delta_0$, that contains $\sigma_1$ and whose support is $\sigma_0$. 
Embed $N_\R$ in $N_\R\oplus \R$ as the hyperplane $\{x_{m+1}=0\}$, let $\tau_0$ be the image of $\sigma_0$, and let $\pi:N_\R\oplus \R\to N_\R$ be the projection. 
For each ray $\R_+ v$ of $\sigma_1$ that is not in $\partial \sigma_0$, choose $t_v\in\R_+$ and let $T$ be the convex hull in $N_\R\oplus\R$ of $\tau_0$ and the rays $\R_+(v,t_v)$. 
Observe that for a generic choice of $t_v$, the faces of $T$ are simplicial cones. 
Let $\Delta_0$ be the collection of images of faces of $T$ except $\sigma_0$ itself. Note that $\pi$ maps $\partial T\setminus \interior \tau_0$ 1-1 onto $\sigma_0$.  It follows that $\Delta_0$ is a simplicial fan, with support equal to $\sigma_0$, and $\sigma_1$ is one of the cones in $\Delta_0$. 
If $\sigma\in\Delta_0$, then either $\sigma$ is a face of $\sigma_0$ or $\interior\sigma\subseteq\interior\sigma_0$. Also, note that all rays in $\Delta_0\setminus\Delta$ are one-dimensional faces of $\sigma_1$. 

\smallskip

We will now construct a fan $\Delta'$ that refines $\Delta$ and contains $\Delta_0$ as a subfan. Let $\Delta_1'$ be the collection of cones in $\Delta$ that do not contain $\sigma_0$. Moreover, let $\Delta_2'$ be the collection of cones of the form $\sigma+\tau$, where $\sigma\in\Delta_0$ and $\tau\in\Delta$ is a face of a cone $\tilde\tau\in\Delta$, such that $\tilde\tau\supseteq\sigma_0$, but $\tau\cap\sigma_0=\{0\}$. Finally, let $\Delta'$ be the union of $\Delta_1'$ and $\Delta_2'$. Note that this union is not disjoint. Observe that all cones in $\Delta'$ are simplicial. We claim that $\Delta'$ is a simplicial fan. 

To prove the claim, first note that if $\sigma\in\Delta$ does not contain $\sigma_0$, then clearly the faces of $\sigma$ do not contain $\sigma_0$. In other words, a face of a cone in $\Delta_1'$ is in $\Delta_1'$. Moreover, a face of a simplicial cone $\sigma+\tau\in\Delta_2'$ is of the form $\sigma'+\tau'$, where $\sigma'$ is a face of $\sigma$ and $\tau'$ is a face of $\tau$. Since $\Delta_0$ is a fan, $\sigma'\in\Delta_0$, and since $\Delta$ is a fan, $\tau'\in\Delta$. Moreover $\{0\}\subseteq \tau'\cap\sigma_0\subseteq\tau\cap\sigma_0=\{0\}$ so $\sigma' +\tau'\in\Delta_2'$. To conclude, a face of a cone in $\Delta'$ is in $\Delta'$. 

It remains to prove that if $\rho$ and $\rho'$ are two distinct cones in $\Delta'$, then $\interior\rho\cap\interior\rho'=\emptyset$. We have to consider three cases. 
In the first case, $\rho,\rho'\in\Delta_1'\subseteq\Delta$. Then $\interior\rho\cap\interior\rho'=\emptyset$, since $\Delta$ is a fan and $\rho\neq\rho'$. 

In the second case, $\rho\in\Delta_2'\setminus\Delta_1'$ and $\rho'\in\Delta_1'$. Then we can write $\rho=\sigma+\tau$, where $\interior\sigma\subseteq \interior\sigma_0$. Indeed, if $\sigma$ is a face of $\sigma_0$, then $\rho\in\Delta_1'$. It follows that $\interior\rho\cap\interior\rho'\subseteq \interior (\sigma_0+\tau)\cap\interior\rho' =\emptyset$. Indeed, $\sigma_0+\tau$ is a cone in $\Delta$, since $\Delta$ is simplicial, and by construction $\rho'\neq \sigma_0+\tau$. 

In the third case, $\rho=\sigma+\tau$ and $\rho'=\sigma'+\tau'$ are both in $\Delta_2'\setminus \Delta_1'$. 
If $\tau\neq \tau'$, then by construction $\sigma_0+\tau$ and $\sigma_0+\tau'$ are two distinct cones in $\Delta$. Hence $\interior\rho\cap\interior\rho'\subseteq \interior(\sigma_0+\tau)\cap\interior(\sigma_0+\tau')=\emptyset$. Next, assume that $\tau = \tau'$. Then $\sigma_0\cap\tau=\{0\}$, which implies that each element $v\in\sigma_0+\tau$ admits a unique representation $v=v_0+w$, where $v_0\in\sigma_0$ and $w\in\tau$. Assume that $v\in\interior \rho\cap\interior\rho'$. Then $v_0\in\interior\sigma\cap\interior\sigma'$, since $\interior\rho=\interior\sigma+\interior\tau$. Hence $\sigma=\sigma'$, which implies that $\rho=\rho'$. To conclude, $\Delta'$ is a simplicial fan.

Now let us show that $\Delta'$ has the desired properties. First, observe that all cones in $\Delta_0$, and in particular $\sigma_1$, are in $\Delta'$. Indeed, if $\sigma\in\Delta_0$, then $\sigma=\sigma+0\in\Delta_2'$. Next, by definition of $\Delta'$, each cone in $\Delta$ that does not contain $\sigma_0$ is in $\Delta_1'\subseteq\Delta'$. Moreover each ray in $\Delta'\setminus\Delta$ is in $\Delta_0\setminus \Delta$ and hence a one-dimensional face of $\sigma_1$. 

It remains to show that $\Delta'$ refines $\Delta$. Consider $\rho\in\Delta$. If $\rho$ does not contain $\sigma_0$, then $\rho$ is itself in $\Delta'$, so assume that $\rho\supseteq\sigma_0$. Since $\Delta$ is simplicial this means that $\rho=\sigma_0+\tau$ for some face $\tau$ of $\rho$, for which $\sigma_0\cap\tau=\{0\}$. Thus 
$
\rho=\sigma_0+\tau=\left (\bigcup_{\sigma\in\Delta_0}\sigma\right )+\tau=
\bigcup_{\sigma\in\Delta_0}\left (\sigma+\tau \right ),
$
and each $\sigma+\tau\in\Delta'_2\subseteq \Delta'$ by construction. 
Hence $\Delta'$ refines $\Delta$. 
\end{proof}

\subsection{Invariant divisors and support functions}\label{divisorstycke}
Let $\cl (X)$ and $\pic (X)$ denote the groups of Weil and Cartier divisors on $X$, respectively, modulo linear equivalence. For $X=X(\Delta)$, $\cl(X)$ and $\pic(X)$ are generated by divisors that are invariant under the action of the torus $T_N$. 
Since $X(\Delta)$ is normal, every Cartier divisor defines a Weil divisor. 

Each ray $\rho$ of $\Delta$ determines a prime Weil divisor $D(\rho)$ that is invariant under the action of $T_N$, and these divisors generate $\cl(X)$ and $\pic(X)$. A $T_N$-invariant Weil divisor is of the form $\sum a_i D(\rho_i)$, where $a_i\in\Z$ and the sum runs over the rays in $\Delta$. 

A $T_N$-invariant Cartier divisor can be represented as a certain piecewise linear  function. We say that $\sf:|\Delta|\to\R$ is a \emph{($\Delta$-linear) $\Q$-support function} if it is linear on each cone of $\Delta$ and if $\sf(|\Delta|\cap N_\Q)\subseteq \Q$. If the restriction of $\sf$ to a cone is given by some element of $M$ (rather than $M_\Q$) then $h$ is said to be a \emph{($\Delta$-linear) support function}. There is a one-to-one correspondence between $T_N$-invariant $\Q$-Cartier divisors and $\Q$-support functions and between $T_N$-invariant Cartier divisors and support functions, see~\cite[Ch~6, p.~6]{Mustata}. Moreover, a $T_N$-invariant $\Q$-Cartier divisor is a Weil divisor if and only if $\sf (|\Delta|\cap N)\subseteq \Z$. Given support functions $h_1$ and $h_2$, the corresponding divisors are linear equivalent if and only if $h_1-h_2$ is linear.

Note that a $\Delta$-linear support function is determined by its values on primitive vectors of rays in $\Delta$. In particular, if $D$ is a $\Q$-Cartier divisor of the form $D=\sum a_i D(\rho_i)$, then the corresponding $\Q$-support function 
is determined by $h(v_i)=a_i$, where $v_i$ is the primitive vector of $\rho_i$. 
Conversely, a support function $h$ determines a 
Weil-divisor $\sum h(v_i) D(\rho_i)$. 

A $\Delta$-linear support function $\sf$ is said to be \emph{strictly convex} if it is convex and defined by different elements $\xi_\sigma\in M$ for each $\sigma\in\Delta(m)$. A compact toric variety $X(\Delta)$ is projective if and only if there is a strictly convex $\Delta$-linear support function, see~\cite[Cor.~2.16]{Oda}. 
We will then say that $\Delta$ is projective. 
\begin{lma}\label{L402}
  Any fan $\Delta$ admits a regular (hence simplicial) 
  projective refinement.
\end{lma}
\begin{proof}
  This is well known, so we only sketch the proof.
  First, by the toric Chow Lemma~\cite[Prop.~2.17]{Oda}
  $\Delta$ admits a projective refinement. In general,
  this refinement is not regular or even simplicial. 
  However, one can check that the standard procedure for
  desingularizing a toric variety by refining the fan
  (see~\cite[\S2.6]{F1}) preserves projectivity.
\end{proof}
\begin{lma}\label{L401}
  If the fan $\Delta$ in Lemma~\ref{barvinok} is projective, 
  then the refinement $\Delta'$ in that lemma can also be chosen projective.
\end{lma}
\begin{proof}
  Assume that $\Delta$ is projective and let $h$ be a strictly convex 
  $\Delta$-linear support function.
  We will show that we can modify $h$ to a strictly convex 
  $\Delta'$-linear function $h'$, where $\Delta'$ is the fan 
  constructed in the proof of Lemma~\ref{barvinok}. 
  We will use the notation from that proof.

  The construction of the fan $\Delta_0$ in the proof of 
  Lemma~\ref{barvinok} implies that there is a strictly convex 
  $\Delta_0$-linear support function $h_0$ that is 
  zero on the boundary of $|\Delta_0|=\sigma_0$. 
  Pick a norm on $M_\R$ and choose 
  $0<\epsilon\ll\min_{\sigma,\tau\in\Delta(m),\sigma\ne\tau}\|\xi_\sigma-\xi_\tau\|$ 
  if $h=\xi_\sigma\in M_\Q$ on $\sigma\in\Delta(m)$. 
  
  Consider $\tau\in\Delta(m)$. Either $\tau\cap\sigma_0=\{0\}$ 
  or $\tau\supseteq \sigma_0$.
  In the first case $\tau\in\Delta'(m)$ and we let $h'=h$ on $\tau$. 
  
  Next, assume that $\tau\supseteq\sigma_0$. 
  Since $\Delta$ is simplicial there exists a unique 
  maximal face $\tau'$ of $\tau$ such that $\tau'\cap\sigma_0=\{0\}$ 
  It follows that an element in $\sigma_0$ 
  admits a unique representation $s+t$, with $s\in\sigma_0$ and $t\in\tau'$. 
  In $\Delta'$, $\tau$ will be subdivided into maximal cones 
  of the form $\rho=\tau'+\sigma$ where $\sigma$ 
  is a cone of maximal dimension in $\Delta_0$. 
  On each $\rho$ let $h'$ be defined by $h'(s+t)=\epsilon h_0(s)$. 
  Clearly $h'$ is linear and strictly convex 
  on the subfan of $\Delta'$ that has support on $\tau$. 
  Moreover, since $h_0$ vanishes on the boundary of $\sigma_0$, the choice of 
  $\epsilon$ ensures that  $h'$ is continuous and convex on $\Delta'$ 
  and different on all $\sigma\in\Delta'(n)$. 
  In other words, $h'$ is $\Delta'$-linear and strictly convex
  and hence $\Delta'$ is projective. 
\end{proof}
If all maximal cones of $\Delta$ are of dimension $m$, then 
$H^2(X(\Delta), \Z)=\pic(X(\Delta))$. 
If $\Delta$ is complete and regular, then $H^{1,1}(X(\Delta))=H^2(X (\Delta), \C)$. 
If $\Delta$ is simplicial, then $X(\Delta)$ is $\Q$-factorial, that is, a Weil-divisor is $\Q$-Cartier.

\section{Monomial maps}\label{sec:monmaps}
Let $\Delta$ and $\Delta'$ be fans in $N\cong \Z^m$ and $N'\cong \Z^{m'}$, respectively. Then any $\Z$-linear map $\phi : N \to N'$ gives rise to a rational map $f: X(\Delta)\dashrightarrow X(\Delta')$, which is equivariant with respect to the actions of $T_N$ and $T_{N'}$. 
Let $e_1, \dots e_m$ and $e'_1,\dots,e'_{m'}$ be bases of $N$ and $N'$,
respectively, and let $x_1,\dots,x_m$ 
and $x'_1,\dots,x'_{m'}$ 
be the corresponding bases for the duals $M$ and $M'$. 
Then $\phi=\sum_{1\leq j\leq m, 1\leq k\leq m'}a_{kj}e_j\otimes x_k'$, where $a_{kj}\in\Z$. Let $z_1,\dots, z_m$ and $z'_1,\dots, z'_{m'}$ be the induced coordinates on $T_N\cong (\C^*)^m$ and $T_{N'}\cong (\C^*)^{m'}$, respectively. 
Then $f:T_N\to T_{N'}$ is given by the monomial map $f: (z_1,\dots, z_m)\mapsto (z_1^{a_{11}}\dots z_m^{a_{1m}}, \dots, z_1^{a_{m'1}}\dots z_m^{a_{m'm}})$. Conversely, any rational, equivariant map $f: X(\Delta)\dashrightarrow X(\Delta')$ comes from a $\Z$-linear map $\phi:N\to N'$, see~\cite[p.19]{Oda}.

The map $f: X(\Delta)\dashrightarrow X(\Delta')$ is holomorphic precisely if the extension $\phi:N_\R\to N'_\R$ satisfies that for each $\sigma\in \Delta$ there is a $\sigma'\in\Delta'$, such that $\phi (\sigma)\subseteq \sigma'$. Let $\phi_{\Delta\Delta'}:(N,\Delta)\to(N',\Delta')$ be the map that takes $(v,\sigma)$ to $(\phi(v), \sigma')$, where $\sigma'$ is the smallest cone in $\Delta'$ that contains $\phi(\sigma)$. If $f$ is holomorphic 
we say that $\phi_{\Delta\Delta'}$ is \emph{regular}. If $f$ is not holomorphic, $\phi_{\Delta\Delta'}$ is not defined everywhere; we write $\phi_{\Delta\Delta'}: (N,\Delta)\dashrightarrow (N', \Delta')$. 
Observe that there is a sub-fan $\widetilde \Delta$ of $\Delta$ which contain all rays of $\Delta$, such that $\phi_{\Delta\Delta'}$ is well-defined on $(N,\widetilde\Delta)$. Indeed, the image of a ray in $\Delta$ under $\phi$ is always contained in a cone in $\Delta'$. 

A $\Z$-linear map $\phi:N\to N'$ induces a $\Z$-linear map 
$\phi^*: M'\to M$ given by  $\phi^*\xi'=\xi'\circ \phi$.

\subsection{Desingularization}
By regularizing fans we can desingularize toric varieties and equivariant maps between toric varieties. 
First, let $\widetilde \Delta$ be a regular refinement of $\Delta$ and let $\text{id}_{\widetilde\Delta\Delta}:(N,\widetilde\Delta)\to (N,\Delta)$ be the map induced by $\text{id}:N\to N$. 
Then the map $\pi: X(\widetilde\Delta)\to X(\Delta)$ induced by 
$\text{id}_{\widetilde\Delta\Delta}$ is a resolution of singularities, see~\cite[Ch.~2.5]{F1}. 

Second, let $N$ and $N'$ be lattices of the same rank, let $\Delta$ and $\Delta'$ be fans in $N$ and $N'$, respectively, and let $\phi :N\to N'$ be a $\Z$-linear map of maximal rank.  We claim that there is a regular refinement $\widetilde \Delta$ of $\Delta$ such that the map $\phi_{\widetilde\Delta\Delta'}:(N,\widetilde\Delta)\dashrightarrow (N',\Delta')$ induced by $\phi:N\to N'$ is regular. We obtain the left-hand diagram 
of~\eqref{lyftning} 
\begin{equation}\label{lyftning}
  \xymatrix{
    &&
    (N,\widetilde\Delta)
    \ar[d]_{\text{id}_{\widetilde\Delta\Delta}}
    \ar[dr]^{\phi_{\widetilde\Delta\Delta'}}
    &&
    X(\tilde\Delta)
    \ar[d]_{\pi}
    \ar[dr]^{\tilde{f}}
    &&
    \\
    &&
    (N,\Delta)
    \ar@{-->}[r]^{\phi_{\Delta\Delta'}}
    &    
    (N',\Delta')
    &
    X(\Delta)
    \ar@{-->}[r]^{f}
    &    
    X(\Delta')
  }
\end{equation}
inducing the right-hand diagram, where
$\pi:X(\widetilde \Delta)\to X(\Delta)$ 
and $\tilde f:X(\widetilde \Delta)\to X(\Delta')$ are holomorphic.

To prove the claim, let $\phi^{-1}(\Delta')$ be the collection of cones $\phi^{-1}(\sigma')$, where $\sigma'\in\Delta'$. Since $\phi$ is of maximal rank, the cones in $\phi^{-1}(\Delta')$ are strictly convex, and thus $\phi^{-1}(\Delta')$ is a fan. Now, any regular fan $\widetilde\Delta$ that refines both $\Delta$ and $\phi^{-1}(\Delta ')$ has the desired properties, and recall from Section~\ref{fansoch} that such a fan always exists.

\subsection{Pullback and pushforward under holomorphic maps}
Let $N$ and $N'$ be lattices of the same rank, let $\Delta$ and $\Delta'$ be fans in $N$ and $N'$, respectively, and let $\phi:N\to N'$ be a $\Z$-linear map of maximal rank, such that $\phi_{\Delta\Delta'}$ is regular. Let $f:X(\Delta)\rightarrow X(\Delta')$ be the corresponding holomorphic map on toric varieties. 

Let $D$ be a $T_{N'}$-invariant $\Q$-Cartier divisor on $X(\Delta')$ and let $\sf_D$ be the corresponding $\Q$-support function. Then the pullback $f^*D$ is a well-defined $\Q$-Cartier divisor, see~\cite[Ch.~2.2]{F2}, and $h_{f^*D}=\phi^* h_D$, see for example~\cite[Ch.~6, Exercise~8]{Mustata}. 
If $D$ is Cartier, then $f^*D$ is Cartier. To see this, assume that $\sf$ is a support function on $\Delta'$. Pick $\sigma\in\Delta$ and assume that $\phi(\sigma)\subseteq \sigma'$.
On $\sigma'\in\Delta'$, $\sf$ is defined by $\sf=\xi'$, for some $\xi'\in M'$. It follows that on $\sigma$, $\phi^* h=\phi^*\xi'\in M$. 

Next, let $D=\sum a_i D(\rho_i)$ be a $T_N$-invariant Weil divisor on $X(\Delta)$. Then $f_* D$ is a well-defined $T_{N'}$-invariant Weil divisor on $X(\Delta')$, see for example~\cite[Ch.~1.4]{F2}, and $f_*D=\sum a_in_i D(\phi(\rho_i))$, where the sum is over all $i$ such that $\phi(\rho_i)\in \Delta'$ and where $n_i\in\N^*$.
Note that the pushforward of a Cartier divisor is 
in general only $\Q$-Cartier. 
Pullback and pushforward respect linear equivalence. 

\subsection{Pullback under rational maps}
Let $N$ and $N'$ be lattices of the same rank, let $\Delta$ and $\Delta'$ be fans in $N$ and $N'$, respectively, and let $\phi:N\to N'$ be a $\Z$-linear map of maximal rank. 
Let $D$ be a $T_{N'}$-invariant Cartier divisor on $X(\Delta')$. 
In terms of the right-hand diagram of~\eqref{lyftning},
we can define the pullback of $D$ under $f$ as $f^*D:=\pi_*\tilde f^* D$. 
In fact this definition does not depend on the particular 
choice of $\widetilde\Delta$. 
Observe that $f^*D$ is the $T_N$-invariant Weil divisor 
$\sum -\sf_D(\phi (v_i)) D(\rho_i)$, where the sum is over 
the rays $\rho_i\in\Delta (1)$. 

Assume that $\Delta$ is simplicial, and let $\sf$ be a $\Delta'$-linear $\Q$-support function. Let $\phi_{\Delta\Delta'}^*h$ be the $\Delta$-linear support function defined by $(\phi_{\Delta\Delta'}^* h)(v)=h(\phi (v))$ if $v$ is a primitive vector of a ray in $\Delta$. 
In other words, $\phi_{\Delta\Delta'}^*h$ is obtained from $\phi^*h$ using 
$\Delta$-linear interpolation.
If $D$ is a $T_{N'}$-invariant $\Q$-Cartier divisor on $X(\Delta')$, and $h_D$ is the corresponding $\Q$-support function, then $f^*D$ is determined by the $\Q$-support function $\phi_{\Delta\Delta'}^*h_D$.

Let $\sigma\in\Delta$ and let $\sf$ be a $\Delta'$-linear $\Q$-support function. Assume that $\phi(\sigma)$ is contained in a cone $\sigma'\in \Delta'$. Then $\sf$ is linear on $\phi(\sigma)$, which implies that $\phi^*\sf$ is linear on $\sigma$ and
\begin{equation}\label{stammer}
\phi^*_{\Delta\Delta'}\sf|_\sigma=\phi^*\sf|_{\sigma}.
\end{equation}
In particular, if $\phi_{\Delta\Delta'}$ is regular, then $\phi_{\Delta\Delta'}^* h=\phi^* h$ for all $\Delta'$-linear $\Q$-support functions. 
This is not the case in general if $\phi_{\Delta\Delta'}$ is not regular. Assume that there are cones $\sigma\in\Delta$ and $\sigma'_1,\sigma'_2\in\Delta'$, such that $(\interior\phi(\sigma))\cap\sigma'_j\neq\emptyset$ for $j=1,2$, and moreover that $\sf|_{\sigma'_1}$ and $\sf|_{\sigma'_2}$ are not defined by the same element in $M'$. Then $\phi^*_{\Delta\Delta'}\sf\neq\phi^*\sf$; indeed $\phi_{\Delta\Delta'}^*\sf$ is linear on $\sigma$, whereas $\sf$ is not linear on $\phi (\sigma)$.

\subsection{Criteria for stability}
We can express the stability of $f:X(\Delta)\dashrightarrow X(\Delta)$ in terms of $\phi:N\to N$. 
\begin{lma}\label{kommuterar}
Assume that $N$, $N'$, and $N''$ are lattices of the same rank, that $\Delta$, $\Delta'$, and $\Delta''$ are complete simplicial fans in $N$, $N'$, and $N''$, respectively, and that 
\[
N\stackrel{\phi}{\rightarrow} N'
\stackrel{\phi'}{\rightarrow} N''
\]
are $\Z$-linear maps of maximal rank, with corresponding rational equivariant maps 
\[
X(\Delta)\stackrel{f}{\dashrightarrow} X(\Delta')
\stackrel{f'}{\dashrightarrow} X(\Delta'').
\]
Moreover, let $f^*$, $f'^*$ be the corresponding pullback maps
\[
\pic(X(\Delta''))\stackrel{f'^*}{\rightarrow} \pic(X(\Delta'))
\stackrel{f^*}{\rightarrow} \pic(X(\Delta)).
\]
For $\rho\in\Delta(1)$ let $\sigma'_\rho$ be the (unique) cone in $\Delta'$ such that $\interior\phi(\rho)\subseteq \interior \sigma'_\rho$. Then  $(f'\circ f)^*=f^*f'^*$ if and only if $\phi'(\sigma'_\rho)$ is contained in a cone in $\Delta''$ for all $\rho\in\Delta(1)$. 
\end{lma}
\begin{proof}
  Note that $f^*f'^*=(f'\circ f)^*$ on $\pic(X(\Delta'')$ 
  if and only if, for every
  $\Delta''$-linear support function $h$, the function
 $\phi_{\Delta\Delta'}^*\phi_{\Delta'\Delta''}^*h-(\phi'\circ\phi)_{\Delta\Delta''}^*h$
  on $N_\Q$ is linear, that is, belongs to $M_\Q$.
  However, this turns 
  out\footnote{We thank Jan-Li Lin for pointing  this out.} 
  to be equivalent to the
  (\textit{a priori} stronger) condition
  $\phi_{\Delta\Delta'}^*\phi_{\Delta'\Delta''}^*h=(\phi'\circ\phi)_{\Delta\Delta''}^*h$
 for every $\Delta''$-linear support function $h$, 
 see~\cite[Prop.~5.5]{Lin}.
 Further, the latter condition can be written 
 $\phi_{\Delta\Delta'}^*\phi_{\Delta'\Delta''}^*h(v)=\phi^*\phi'^*h(v)$
 for all primitive vectors of rays of $\Delta$. 
 
 Let $\rho$ be a ray of $\Delta$ with corresponding 
 primitive vector $v$ and let $\sigma_\rho'\in\Delta'$ be the (unique)
 cone for which $\phi(v)\in\interior\sigma_\rho'$.

 First assume there is a cone $\sigma''\in\Delta''$
 such that $\phi'(\sigma_\rho')\subseteq\sigma''$. 
 and let $\sf$ be a $\Delta''$-linear support function. Then 
 \begin{equation*}
   \phi_{\Delta\Delta'}^*\phi'^*_{\Delta'\Delta''}\sf(v)
   =\phi^*(\phi'^*_{\Delta'\Delta''}\sf)|_{\sigma_\rho'}(v)
   =\phi^*(\phi'^*\sf)|_{\sigma_\rho'}(v)
   =\phi^*\phi'^*\sf(v); 
 \end{equation*}
 here we have used that $v$ is a primitive vector 
 of $\rho$ for the first equality and \eqref{stammer}
 for the second equality. 
 Hence the ``if''-direction of the lemma follows. 
 
 Now assume $\phi'(\sigma_\rho')$ is not contained 
 in any cone in $\Delta''$.  It follows that there are 
 cones $\sigma_1'', \sigma_2''\in\Delta''$, 
 such that 
 $\dim (\phi'(\sigma'_\rho)\cap \sigma''_j)=\dim\sigma'_\rho$ 
 for $j=1,2$. 
 Note that then $\sigma''_1\not\subseteq\sigma_2''$. 
 Pick $\rho''_1\in\Delta''(1)$, such that $\rho''_1$ 
 is a face of $\sigma''_1$ but not of $\sigma''_2$, 
 and let $v_1''$ be the corresponding primitive vector. 
 Let $\sf$ be the $\Delta''$-linear
 $\Q$-support function that is determined by 
 $\sf(v_1'')=1$, but $\sf(v_i'')=0$ for all other 
 primitive vectors of rays in $\Delta''$. 
 Then $\sf\not\equiv 0$ on $\sigma''_2$, 
 but $h\equiv 0$ on $\sigma''_1$, which implies that  
 $\phi'^*_{\Delta'\Delta''}\sf$ is linear on $\sigma'_\rho$, 
 whereas $\phi'^*\sf$ is not. In particular,
 $\phi'^*_{\Delta\Delta'}\sf(\phi(v)) > \phi'^*\sf(\phi(v))$, 
 since $\phi(v)\in\interior\sigma'_\rho$.
 Hence 
 $\phi_{\Delta\Delta'}^*\phi'^*_{\Delta'\Delta''}h(v)\neq
 \phi^*\phi'^*h(v)$,
 proving the ``only if''-direction of the lemma.
\end{proof}

The following results are immediate consequences of Lemma~\ref{kommuterar}.

\begin{cor}\label{stablma}
Assume that $\Delta$ is a complete simplicial fan in $N$ and that 
$\phi:N\to N$ is a $\Z$-linear map, with corresponding rational equivariant map 
$f:X(\Delta)\dashrightarrow X(\Delta)$. 

Then $f$ is $1$-stable if and only if all cones $\sigma$ in $\Delta$ for which there is a ray $\rho\in\Delta(1)$ such that $\interior \phi^n(\rho)\subseteq\interior\sigma$ for some $n$, satisfy that $\phi^{n'}(\sigma)$ is contained in a cone in $\Delta$ for all $n'\in\N$.
\end{cor}

When $\phi:N\to N$ satisfies the assumption in Corollary~\ref{stablma} we say that it is \emph{torically stable} on $\Delta$. 

\begin{cor}\label{invkoner}
  Assume that $\Delta$ is a complete simplicial fan and 
  $\phi:N\to N$ a $\Z$-linear map.
  Assume there is a collection $\cS\subseteq\Delta$ such that 
  $\phi(\sigma)\subseteq\sigma$ for 
  $\sigma\in\cS$ and such that each ray in $\Delta$ is 
  either mapped onto another ray in $\Delta$ or mapped 
  into one of the cones in $\cS$.
  Then $f:X(\Delta)\dashrightarrow X(\Delta)$ is 1-stable. 
\end{cor}

\begin{remark}\label{teckenrmk}
  In Corollary~\ref{invkoner} it suffices to require that $\phi$
  maps every cone in $\cS$ into some other cone in $\cS$.
\end{remark}

In Theorems~A and~B we require that the eigenvalues of $\phi$ be
positive. The reason is that when some eigenvalues are negative,
it may be impossible to find invariant cones:
\begin{prop}\label{spar}
  Let $\phi:N\to N$ be a $\Z$-linear map, where $N\cong \Z^m$. 
  If $\sigma$ is a simplicial cone of dimension $m$ such that 
  $\phi(\sigma)\subseteq\sigma$, then the trace 
  of $\phi$, i.e.,\ the sum of the eigenvalues, must be nonnegative.
\end{prop}
This result is presumably known, but we include a proof
for completeness.
\begin{proof}
  Let $v_1,\dots,v_m$ be a basis for $N_\R$ such that
  $\sigma=\sum_{j=1}^m\R_+ v_j$. 
  We may assume $\det(v_1,\dots,v_m)=1$.
  Since $\phi(\sigma)\subseteq\sigma$, we must have
  \begin{equation}\label{e204}
    \det(v_1,\dots,v_{j-1},\phi(v_j),v_{j+1},\dots,v_m)\ge0
    \quad\text{for $1\le j\le m$}.
  \end{equation}
  As the left hand side of~\eqref{e204} equals the $j$th
  diagonal element in the matrix of $\phi$ in the chosen basis,
  the lemma follows by summing~\eqref{e204} over $j$.
\end{proof}
%
%
%
%
%
%
\section{Smooth stabilization in dimension two}\label{sec:dimtwo}
We now look at two-dimensional monomial mappings. 
Recall that a $\Z$-linear map
$\phi:N\to N$ is torically stable
on a fan $\Delta$ if it satisfies the condition in Corollary~\ref{stablma}.

Note that the 
eigenvalues $\mu_1$, $\mu_2$ are either both real or complex 
conjugates of each other.
When they are real, they are either both integers or
both irrational.

In~\cite{Favre}, Favre gave a complete characterization of 
when we can make $\phi$ torically stable 
on a possibly irregular fan, see Theorem~C' in the introduction.
For the rest of this section we analyze whether we can make 
$\phi$ torically stable on a \emph{regular} fan.
We will prove Theorem~C and give several examples
showing that this result is essentially sharp. 
We also recover Theorem~C'.
The main new ideas are contained in Section~\ref{sec:S301}.
\begin{remark}\label{R201}
  The statement in~\cite[Th\'eor\`eme Principal]{Favre} 
  does deal with smooth toric varieties, but there is 
  a gap in the proof. Suppose 
  $\mu_2/\mu_1$ is not of the form $e^{i\pi\theta}$ with 
  $\theta\in\R\setminus\Q$. What Favre proves is that 
  we can find a (not necessarily regular)
  refinement $\Delta'$ of $\Delta$ on which
  $\phi$ is torically stable.
  He then asserts that $\Delta'$ becomes a regular fan by
  passing to a sublattice $N'\subseteq N$. However, 
  this last step does not work in general, 
  see~\cite[\S2.2, p.36]{F1}.
\end{remark}
%
%
%
%
\subsection{Basic facts on fans in dimension two}\label{facts2}
We need a few basic results about refinements of fans
in dimension two.
Let us call a fan $\Delta$  \emph{symmetric} 
if $-\sigma\in\Delta$ for every cone $\sigma\in\Delta$.

First, as described in~\cite[Section~2.6]{F1} there
exists a canonical procedure for making an irregular fan
regular. In fact, we have
\begin{lma}\label{lma:minrefine}
  Every fan $\Delta$ admits a regular refinement $\Delta'$ such that:
  \begin{itemize}
  \item[(i)]
    any regular cone in $\Delta$ is also a cone 
    in $\Delta'$;
  \item[(ii)]
    if $\Delta$ is symmetric, then so is $\Delta'$.
  \end{itemize}
\end{lma}
Both the lemma and its proof  below are valid in any dimension.
\begin{proof}[Sketch of proof]
  The construction of $\Delta'$ proceeds inductively as follows.
  Pick an irregular two-dimensional cone $\tau$ in $\Delta$,
  let $\sigma_1,\sigma_2$ be its one-dimensional faces
  and let $v_j\in\sigma_j\cap N$ be the associated primitive vectors.
  Since $\tau$ is irregular, there exists 
  $t_i\in(0,1)\cap\Q$, $i=1,2$, 
  such that $v:=t_1v_1+t_2v_2\in N$.
  Let $\sigma=\R_+v$ and let
  $\tau_i$, $i=1,2$ be the two-dimensional cones
  whose one-dimensional faces are $\sigma_i$ and $\sigma$.
  Applying this procedure finitely many times yields a regular
  fan $\Delta'$. 
  If $\Delta$ is symmetric, then we may simultaneously subdivide
  $\tau$ and $-\tau$ into $\tau_1$, $\tau_2$ and
  $-\tau_1$, $-\tau_2$, respectively. In this way, $\Delta'$ will
  stay symmetric.
\end{proof}
Second, we need to refine fans that may already be regular.
Let $\tau$ be a regular two-dimensional cone
and $\sigma_1,\sigma_2$ its one-dimensional faces with corresponding primitive vectors $v_1$ and $v_2$. Then $v_1,v_2$ generate $N$. 
Let $\sigma=\R_+(v_1+v_2)$ and let
$\tau_i$, $i=1,2$ be the two-dimensional cones
whose one-dimensional faces are $\sigma_i$ and $\sigma$.
The \emph{barycentric subdivision} of $\tau$ consists of
replacing $\tau$ by $\tau_1$ and $\tau_2$.
\begin{remark}
  Both the barycentric subdivision and the closely
  related procedure in the proof of Lemma~\ref{lma:minrefine}
  are special cases of Lemma~\ref{barvinok}.
\end{remark}
\begin{lma}\label{lma:subdivide2}
  Let $(\tau_n)_{n\ge0}$ be a sequence of 
  regular two-dimensional cones such that 
  $\tau_{n+1}$ is one of the two cones obtained
  by barycentric subdivision of $\tau_{n}$, $n\ge0$.
  Then $\bigcap_{n\ge0}\tau_n=\R_+w$ for some $w\in N_\R$. 
\end{lma}
\begin{proof}
  Since $\tau_0$ is regular, 
  $\tau_0=\R_+v+\R_+v'$, where $v, v'$ generate $N$.
  Write $\tau_n=\R_+v_n+\R_+v'_n$, where $v_n, v'_n$ 
  generate $N$. We can assume  
  $v_{n+1}=v_n+v'_n$ and $v'_{n+1}\in\{v_n,v'_n\}$.
  Inductively, we see that 
  $v_n=p_nv+q_nv'$ and $v'_n=p'_nv+q'_nv'$,
  where $p_n,q_n,p'_n,q'_n\ge0$ and 
  $|p_nq'_n-p'_nq_n|=1$. 
  The lemma follows since 
  $\max\{p_n,q_n\}\to\infty$ or 
  $\max\{p'_n,q'_n\}\to\infty$ as $n\to\infty$.
\end{proof}
\begin{lma}\label{lma:subdivide1}
  Consider regular two-dimensional cones 
  $\tau,\tau_0\subseteq N_\R$ 
  such that $\tau\subsetneq\tau_0$. 
  Then $\tau$ is obtained from $\tau_0$ 
  by performing finitely many barycentric subdivisions.
\end{lma}
\begin{proof}
  By Lemma~\ref{lma:subdivide2} and induction
  it suffices to show that 
  $\tau$ is contained in one of the 
  two-dimensional cones obtained by 
  barycentric subdivision of $\tau_0$.

  Write $\tau_0=\R_+v_1+\R_+v_2$ and 
  $\tau=\R_+w_1+\R_+w_2$ where $v_1,v_2$ and 
  $w_1,w_2$ are generators of $N$. 
  As $\tau$ is regular and $\tau\subseteq\tau_0$ 
  we may assume $w_i=p_iv_1+q_iv_2$, where 
  $p_i,q_i\ge0$, $i=1,2$, $p_1q_2-p_2q_1=1$, $p_1>p_2$, and thus $q_2>q_1$. 

It suffices to prove that $v_1+v_2\not\in\interior\tau$ unless $\tau=\tau_0$. 
Assume that $v_1+v_2\in\interior\tau$.  Then 
  $p_1>q_1$ and $p_2<q_2$, which implies that $p_1q_2-p_2q_1 \geq q_1+p_2+1$. It follows that $q_1=p_2=0$, and hence $\tau=\tau_0$. 
\end{proof}
\begin{cor}\label{cor:regular}
  Let $\tau_n$, $n\ge 0$ be regular two-dimensional cones such that 
  $\tau_{n+1}\subsetneq\tau_n$ for all $n$.
  Then $\bigcap_{n\ge0}\tau_n=\R_+w$ for some $w\in N_\R$. 
\end{cor}
%
%
%
%
\subsection{The case $|\mu_1|>|\mu_2|$}
We consider first the case when $|\mu_1|>|\mu_2|$. Then the $\mu_i$ are real 
and the corresponding eigenspaces $E_i\subseteq N_\R$ are one-dimensional. Either $\mu_1,\mu_2\in\Z$ or $\mu_1,\mu_2\notin \Q$. 
As $n\to \infty$ we have $\phi^n(v)\to E_1$ for any $v\in N_\R\setminus E_2$
and $\phi^{-n}(v)\to E_2$ for any $v\in N_\R\setminus E_1$.

We will use the following criterion for making $\phi$ torically stable.
\begin{lma}\label{lma:criterion1}
  Suppose that $\Delta$ is a regular fan 
  and define $U_i\subseteq N_\R$ as the 
  union of all cones in $\Delta$ that 
  intersect $E_i\setminus\{0\}$ for $i=1,2$.
  Assume that $\phi(U_1)\subseteq U_1$ and $\phi^{-1}(U_2)\subseteq U_2$.
  Then there exists a regular refinement $\Delta'$ of $\Delta$ 
  on which $\phi$ is torically stable. 
  If $\Delta$ is symmetric, then we can choose 
  $\Delta'$ symmetric.

  Conversely, suppose $\phi$ is torically stable with respect to 
  a regular fan $\Delta$ and define $U_i$ as above.
  Then $\phi(U_1)\subseteq U_1$ and $\phi^{-1}(U_2)\subseteq U_2$.
\end{lma}
\begin{proof}
  We may assume $U_1\cup U_2\ne N_\R$, or else $f:X(\Delta)\dashrightarrow X(\Delta)$ is $1$-stable by Corollary~\ref{invkoner}. 
  We define an integer $J\ge1$ and a sequence of 
  (not necessarily convex) cones
  \begin{equation*}
    \Omega_0=U_2, \Omega_1,\dots,\Omega_J,\Omega_{J+1}=U_1
  \end{equation*}
  as follows. 
  The set $\Omega_1:=\phi(\Omega_0)\setminus(U_1\cup U_2)$ 
  is nonempty and there exists $J\ge1$ minimal such that 
  $\phi^J(\Omega_1)\subseteq U_1$.
  Set $\Omega_j=\phi^{j-1}(\Omega_1)\setminus U_1$ for $1<j\le J$.
  Then $\{\Omega_j\setminus\{0\}\}_{j=0}^{J+1}$ defines a partition of 
  $N_\R\setminus\{0\}$.
  Note that $\phi(\Omega_j)\subseteq\Omega_{j+1}\cup U_1$ 
  for $1\le j\le J$.

  Let $\Delta_1$ be the fan obtained from $\Delta$ by 
  adding all rays of the form $\phi(\sigma)\in\Omega_1$, 
  where $\sigma$ is a ray in $\Delta$ contained in $U_2$.
  Let $\Delta'_1$ be a regular refinement of $\Delta_1$, in which the regular cones of $\Delta_1$ are kept, as described in Lemma~\ref{lma:minrefine}. 
  Note that this refinement procedure does not subdivide 
  any cone contained in $U_1\cup U_2$.

  Inductively, for $1<j\le J$,
  let $\Delta_j$ be the fan obtained from $\Delta'_{j-1}$ by
  adding all rays of the form $\phi(\sigma)\in\Omega_j$,
  where $\sigma$ is a ray in $\Delta'_{j-1}$ contained in 
  $\Omega_{j-1}$, and let $\Delta'_j$ be the 
  regular refinement of $\Delta_j$ given by Lemma~\ref{lma:minrefine}.
  This refinement procedure does not modify any cone contained in 
  $U_1\cup U_2\cup\Omega_1\cup\dots\cup\Omega_{j-1}$.
  Then we can use $\Delta'=\Delta'_J$.
  If $\Delta$ is symmetric, then so is $\Delta'$. 

  For the second part of the lemma, 
  note that if $\sigma$ is a ray in $\Delta$ that is
  not contained in $E_1\cup E_2$, then 
  $\phi^n(\interior\sigma)\subseteq\interior U_1$ for $n\gg1$. 
  Thus, for $\phi$ to be torically stable on $\Delta$,
  $\phi$ must map any cone contained 
  in $U_1$ into another cone contained in $U_1$. This implies
  $\phi(U_1)\subseteq U_1$.
  A similar argument shows $\phi^{-1}(U_2)\subseteq U_2$.
\end{proof}
%
%
\subsubsection{Integer eigenvalues}\label{S101}
The first subcase is when $|\mu_1|>|\mu_2|$ and $\mu_1,\mu_2\in\Z$ and thus the corresponding eigenspaces $E_1, E_2$ are rational. 
We claim that $\phi$ can always be made torically stable on a regular fan in this case.
To see this, we only need to satisfy the hypotheses of 
Lemma~\ref{lma:criterion1}. After refining, we may assume that $\Delta$
is symmetric and regular and that the eigenspaces $E_i\subseteq N_\R$ 
are unions of cones in $\Delta$. 

For $i=1,2$, let $U_i$ be the union of 
all cones in $\Delta$ intersecting $E_i$. 
If $\mu_1,\mu_2>0$, then $\phi(U_1)\subseteq U_1$ and
and $\phi^{-1}(U_2)\subseteq U_2$.
Hence Lemma~\ref{lma:criterion1} applies.
The same is true also when $\mu_1,\mu_2<0$ since 
$\Delta$ is symmetric.

When $\mu_1$ and $\mu_2$ have opposite signs, we have to be more 
careful. For definiteness, let us assume $\mu_1>0>\mu_2$. 
(The case $\mu_1<0<\mu_2$ is handled the same way as long as
all fans we construct are symmetric.)
Let $\sigma_1$ and $\sigma_2$ be two-dimensional cones 
in $\Delta$ sharing a common face $\tau$ contained in $E_1$.
Provided $\Delta$ is symmetric, $\phi(U_1)\subseteq U_1$ is equivalent to 
$\phi(\sigma_1)\subseteq\sigma_2$ and
$\phi(\sigma_2)\subseteq\sigma_1$, which will only happen
if $\sigma_1$ and $\sigma_2$ have roughly the same size.
We claim this can be arranged by subdividing the cones $\sigma_i$.
Pick generators $v_1$, $v_2$ for $N$ such that 
$v_1\in\tau=\sigma_1\cap\sigma_2$.
Then $\phi$ is given by the matrix
$\bigl( \begin{smallmatrix} 
  a&b\\ 0&-d 
\end{smallmatrix} \bigr)$
where $a=\mu_1>0$ and $0<d=|\mu_2|<a$.
The one-dimensional faces
of $\sigma_1$ (resp.\ $\sigma_2$)
are $\tau$ and a ray whose primitive
vector is of the form $r_1v_1+v_2$ 
(resp. $r_2v_1-v_2$), where $r_1,r_2\in\Z$.
By making a barycentric subdivision of $\sigma_i$ 
and replacing $\sigma_i$ by the subcone containing 
$\tau$ we replace $r_i$ by $r_i+1$.
Repeating this procedure finitely many times, we can
achieve $r_1=r_2=r\gg0$.
Picking $r>|b|/(a-d)$ it is straightforward to verify that 
$\phi(\sigma_1)\subseteq\sigma_2$
and $\phi(\sigma_2)\subseteq\sigma_1$.
Making the construction symmetric, we obtain 
$\phi(U_1)\subseteq U_1$.
A similar construction gives 
$\phi^{-1}(U_2)\subseteq U_2$.
Thus Lemma~\ref{lma:criterion1} applies.
%
%
\subsubsection{Irrational eigenvalues}\label{sec:S301}
The second subcase is when $|\mu_1|>|\mu_2|$ and 
$\mu_1$ and $\mu_2$ are both real irrational.
Then the corresponding eigenspaces $E_i\subseteq N_\R$, $i=1,2$,
contain no nonzero lattice points.
\begin{prop}\label{prop:dim2stab1}
  If $\mu_1,\mu_2$ are of the same sign then any fan 
  $\Delta$ admits a regular refinement $\Delta'$ on which 
  $\phi$ is torically stable.
\end{prop}
\begin{proof}
  We may assume $\Delta$ is symmetric.
  The assumption that $\mu_1$ and $\mu_2$ have  the same
  sign implies that \emph{any} symmetric cones $U_i$, $i=1,2$
  for which $E_i\setminus\{0\}\subseteq\interior U_i$ must satisfy
  $\phi(U_1)\subseteq U_1$ and $\phi^{-1}(U_2)\subseteq U_2$.
  Thus the proposition follows from Lemma~\ref{lma:criterion1}.
\end{proof}

Now assume $\mu_1$ and $\mu_2$ have different signs.
This case is quite delicate.
Let us assume for now that $\mu_1>0>\mu_2$.

Our starting point is a regular two-dimensional cone $\sigma_1$
containing an eigenvector associated to $\mu_1$
but not containing any eigenvector associated
to $\mu_2$. 
Such a cone exists and can be constructed
using repeated barycentric subdivisions and 
invoking Lemma~\ref{lma:subdivide2}.
Write $\sigma_1=\R_+v_1+\R_+v_2$, where $v_1,v_2$
are generators for $N$. 
Then $\phi$ admits eigenvectors of the form
$v_1+z_iv_2$ associated to $\mu_i$, $i=1,2$,
where $z_2<0<z_1$.
After exchanging $v_1$ and $v_2$ if necessary, 
we may and will assume that $\max\{|z_1|,|z_2|\}>1$.

We now inductively 
define a sequence $(v_{1,n},v_{2,n})_{n\in\Z}$ 
of generators for $N$. They will have the property that 
$\phi$ admits an eigenvector of the form
$v_{1,n}+z_{i,n}v_{2,n}$ associated to $\mu_i$, $i=1,2$,
where $z_{2,n}<0<z_{1,n}$ and $\max\{|z_{1,n}|,|z_{2,n}|\}>1$.
Set $v_{i,0}:=v_i$ and $z_{i,0}=z_i$, $i=1,2$.

First suppose $n>0$. 
If $z_{1,n-1}>1$, set
$(v_{1,n},v_{2,n}):=(v_{1,n-1}+v_{2,n-1},v_{2,n-1})$
and if $0<z_{1,n-1}<1$, set $(v_{1,n},v_{2,n}):=(v_{2,n-1},v_{1,n-1})$.
This leads to
\begin{equation}\label{e201}
  (z_{1,n},z_{2,n})
  =\begin{cases}
  (z_{1,n-1}-1,z_{2,n-1}-1)  &\text{if $z_{1,n-1}>1$}\\
  (z_{1,n-1}^{-1},z_{2,n-1}^{-1}) &\text{if $0<z_{1,n-1}<1$}
  \end{cases}.
\end{equation}
Now suppose $n<0$.
If $z_{2,n+1}<-1$, set
$(v_{1,n},v_{2,n}):=(v_{1,n+1}-v_{2,n+1},v_{2,n+1})$
and if $-1<z_{2,n+1}<0$, set $(v_{1,n},v_{2,n}):=(-v_{2,n+1},-v_{1,n+1})$.
We obtain
\begin{equation}\label{e202}
  (z_{1,n},z_{2,n})
  =\begin{cases}
  (z_{1,n+1}+1,z_{2,n+1}+1)  &\text{if $z_{2,n+1}<-1$}\\
  (z_{1,n+1}^{-1},z_{2,n+1}^{-1}) &\text{if $-1<z_{2,n+1}<0$}.
  \end{cases}
\end{equation}

Notice that~\eqref{e201} and~\eqref{e202} in fact hold
for \emph{all} $n\in\Z$. This follows from the fact that
$\max\{|z_{1,n}|,|z_{2,n}|\}>1$.
For example, suppose $n\le0$ and that $z_{1,n-1}>1$.
To verify~\eqref{e201} we must show that
$(z_{1,n},z_{2,n})=(z_{1,n-1}-1,z_{2,n-1}-1)$.
This follows from~\eqref{e202} applied to $n-1$
if we know that $z_{2,n}<-1$.
But if $-1<z_{2,n}<0$, then $z_{1,n}>1$ and
so~\eqref{e202} would give 
$z_{1,n-1}=z_{1,n}^{-1}<1$, a contradiction. 

For any $n\in\Z$,
$\sigma_{1,n}:=\R_+v_{1,n}+\R_+v_{2,n}$ 
and $\sigma_{2,n}:=\R_+v_{1,n}+\R_+(-v_{2,n})$ 
are regular cones containing eigenvectors 
associated to $\mu_1$ and $\mu_2$, respectively, in their interiors.
For $n>0$, $\sigma_{1,n}$ is obtained by barycentric subdivision 
of $\sigma_{1,n-1}$.
For $n<0$, $\sigma_{2,n}$ is obtained by barycentric subdivision 
of $\sigma_{2,n+1}$.
This implies that the sequences $(\sigma_{i,n})_{n\in\Z}$, 
$i=1,2$ are largely independent of the initial choice of 
cone $\sigma_1$.
Indeed, suppose we start with another cone $\sigma'_1$,
obtaining corresponding sequences $(\sigma'_{i,n})_{n\in\Z}$.
By Lemma~\ref{lma:subdivide1} there exist
$l_i\in\Z$, $i=1,2$, such that $\sigma'_{1,n}=\sigma_{1,n+l_1}$ 
and $\sigma'_{2,n}=\sigma_{2,n+l_2}$ for
$n\gg0$ and $n\ll0$, respectively.

Let 
$A_n=\bigl( \begin{smallmatrix} 
  a_n&b_n\\ c_n&d_n 
\end{smallmatrix} \bigr)$
be the matrix of $\phi$ in the basis $(v_{1,n},v_{2,n})$.
We are interested in whether $A_n$ has nonnegative entries.
A direct computation shows 
$b_n=(\mu_1-\mu_2)/(z_{1,n}-z_{2,n})>0$ and
$c_n=(\mu_1-\mu_2)/(z_{1,n}^{-1}-z_{2,n}^{-1})>0$
for all $n$.
As for the diagonal entries, 
note that $a_n+d_n=\mu_1+\mu_2=:\gamma>0$ is
independent of $n$.
Set $\delta_n=a_n-d_n$.
We see that $A_n$ has nonnegative entries 
if and only if $|\delta_n|\le\gamma$.
\begin{lma}\label{lma:criterion2}
  The sequence $(A_n)_{n\in\Z}$ is periodic.
  Further, the following conditions are equivalent:
  \begin{itemize}
  \item[(i)]
    there exists $n$ such that $a_n,b_n,c_n,d_n\ge0$;
  \item[(ii)]
    there exist infinitely many $n\ge0$ such that $a_n,b_n,c_n,d_n\ge0$;
  \item[(iii)]
    any fan $\Delta$ admits a regular refinement $\Delta'$ on which 
    $\phi$ is torically stable.
  \end{itemize}
\end{lma}    
\begin{proof}
  Note that $z_{i,n}$, $i=1,2$ are the roots of
  $b_nz^2+\delta_nz-c_n=0$. 
  It follows from~\eqref{e201} that 
  \begin{equation}
    (b_{n+1},\delta_{n+1},c_{n+1})
    =\begin{cases}
      (b_n,\delta_n+2b_n,c_n-b_n-\delta_n)  &\text{if $c_n>b_n+\delta_n$}\\
      (c_n,-\delta_n,b_n) &\text{if $c_n<b_n+\delta_n$}
    \end{cases}\label{e203}
  \end{equation}
  We see that  the quantity $D\=D_n=\delta_n^2+4b_nc_n$
  is independent of $n$; in fact, $D$ is the discriminant of $b_nz^2+\delta_nz-b_n$. As $b_n,c_n,\delta_n$ are integers and 
  $b_n,c_n>0$, it follows that the sequence 
  $(b_n,\delta_n,c_n)_{n\in\Z}$,
  and hence also the sequence $(A_n)_{n\in\Z}$,
  must be periodic. 
  This immediately shows that~(i) and~(ii) are equivalent.
  
  Before showing that~(i) and~(ii) are equivalent to~(iii), 
  recall that the data constructed so far is essentially 
  independent of the initial choice of regular cone $\sigma_1$.
  In particular, the sequence $(A_n)_{n\in\Z}$ is 
  independent of this choice, up to an index shift, and
  so the validity of~(ii) is independent of $\sigma_1$.

  To show that~(ii) implies~(iii), suppose $A_n$ has 
  nonnegative entries.
  Then the regular cone $\sigma_{1,n}:=\R_+v_{1,n}+\R_+v_{2,n}$ 
  is invariant: $\phi(\sigma_{1,n})\subseteq\sigma_{1,n}$.
  Similarly, the regular cone 
  $\sigma_{2,n}:=\R_+v_{1,n}+\R_+(-v_{2,n})$ 
  satisfies $\phi^{-1}(\sigma_{2,n})\subseteq\sigma_{2,n}$.

  Pick $n_1\gg0$ and $n_2\ll0$ such that 
  $A_{n_i}$ has nonnegative entries for $i=1,2$.
  We may assume $\sigma_{1,n_1}$ 
  and $\sigma_{2,n_2}$ are arbitrarily small
  regular cones containing eigenvectors associated to 
  $\mu_1$ and $\mu_2$, respectively.
  By Lemma~\ref{lma:subdivide1}, we may, after 
  replacing $\Delta$ by a suitable 
  symmetric regular refinement,
  assume that $\pm\sigma_{1,n_1}$ and $\pm\sigma_{2,n_2}$ 
  are cones in $\Delta$. 
  We may then apply Lemma~\ref{lma:criterion1} to
  $U_i=\sigma_{i,n_i}\cup(-\sigma_{i,n_i})$ 
  and conclude that~(iii) holds.

  Finally, to show that~(iii) implies~(i), assume that 
  $\phi$ is torically stable on a regular refinement 
  $\Delta'$ of $\Delta$. 
  We can then use as our initial cone $\sigma_1$
  a cone in $\Delta'$ containing an eigenvector 
  associated to $\mu_1$.
  Indeed, it follows from Lemma~\ref{lma:criterion1}
  that $\phi(\sigma_1)\subseteq\sigma_1$ and that
  $\sigma_1$ cannot contain any eigenvector associated to
  $\mu_2$. The fact that $\phi(\sigma_1)\subseteq\sigma_1$
  implies that $A_0$ has nonnegative entries.
\end{proof}
\begin{remark}
  The sequence $(z_{1,n})_{n\ge0}$ encodes the 
  continued fractions expansion of $z_1$, and the proof that
  $(A_n)_{n\ge0}$ is periodic corresponds to  the classical 
  proof of the (pre)periodicity of the continued 
  fractions expansion of a quadratic surd
  (a result due to Lagrange, see~\cite[Theorem 177, p.185]{HardyWright}).
\end{remark}
\begin{prop}\label{P101}
  When $|\mu_2|<1$,
  any fan $\Delta$ admits a regular refinement $\Delta'$ 
  on which $\phi$ is torically stable.
\end{prop}
\begin{proof}
  Note that $\mu_1$ and $\mu_2$ must be real irrational
  and that $|\mu_1|>1>|\mu_2|$.
  By Proposition~\ref{prop:dim2stab1} we 
  may assume they have different signs.
  First suppose $\mu_1>0>\mu_2$. 
  The condition $-1<\mu_2<0$ easily translates into 
  $\sqrt{D}-2<\gamma<\sqrt{D}$, 
  where $\gamma=\mu_1+\mu_2$ and $D$ is as in the proof of 
  Lemma~\ref{lma:criterion2}; 
  indeed $\mu_j=(\gamma\pm\sqrt D)/2$.
  As noted above, $b_n,c_n>0$ for all $n$
  and by Lemma~\ref{lma:criterion2} we only need to find 
  $n\in\N$ such that $|\delta_n|\le\gamma$,
  where $\delta_n=a_n-d_n$.

  First suppose there exists $n$ 
  such that $c_n=1$ and $|\delta_n|=b_n$. 
  Then $D=\delta_n^2+4b_nc_n=(b_n+2)^2-4$, so
  $\Z\ni\gamma>\sqrt{D}-2$ implies
  $\gamma>b_n+2-2=b_n=|\delta_n|$ and we are done.
  
  In general, it suffices to find $n$
  with $|\delta_n|\le\sqrt{D}-2$, a condition
  equivalent to $|\delta_n|<b_nc_n$.
  There exists $n_0\ge0$ such that $\delta_{n_0}<0$,
  or else we would be able to apply the first
  transformation in~\eqref{e203} infinitely many
  times in a row, which is clearly not possible. Indeed, the second transformation changes the sign of $\delta_n$. 
  Successively applying~\eqref{e203}
  we find $n\ge n_0$ with $-b_n\le\delta_n<b_n$.
  Then $|\delta_n|<b_nc_n$, unless $\delta_n=-b_n$ 
  and $c_n=1$, a case we have already taken care of.

  Finally, consider the case $\mu_1<0<\mu_2$. By what precedes, we
  can find a \emph{symmetric} regular refinement $\Delta'$ 
  of $\Delta$ on which the map $-\phi:N\to N$ is torically stable.
  Then $\phi$ is also torically stable on $\Delta'$.
\end{proof}
The following example shows that Theorem~C fails in general when $|\mu_1|>|\mu_2|>1$. 
\begin{ex}\label{motex1}
  It follows from Lemma~\ref{lma:criterion2} that
  the linear map $\phi:N\to N$ given by the matrix 
  $A=A_\phi=\bigl( \begin{smallmatrix} 
    -1&3\\ 3&2 
  \end{smallmatrix} \bigr)$
  cannot be made torically stable for any complete 
  regular fan.
  Indeed, $A_0=A$,
  $A_1=\bigl( \begin{smallmatrix} 
    2&3\\ 3&-1 
  \end{smallmatrix} \bigr)$
  and $A_n=A_{n-2}$. 
  Here $\mu_j=(1\pm3\sqrt 5)/2$. 
\end{ex}
We record the following consequence of our analysis in Section~\ref{sec:S301}.
\begin{cor}
  Assume that the eigenvalues of $\phi:N\to N$ 
  satisfy $\mu_1>-\mu_2>0$ and $\mu_i\not\in\Z$ for $i=1,2$. 
  Moreover, assume $N$ has generators $v_1,v_2$ such that 
  $\phi$ is given by a matrix with nonnegative coefficients
  in the associated basis for $N_\R$. 
  Then $\phi$ admits an eigenvector $e_1=v_1+z_1v_2$
  in the first quadrant $\sigma_0=\R_+v_1+\R_+v_2$
  and there exists a sequence 
  $(\sigma_j)_{j\ge0}$ of regular cones such that 
  $\R_+e_1\subseteq\sigma_{j+1}\subseteq\sigma_j$,
  $\bigcap_{j=0}^\infty\sigma_j=\R_+e_1$ and $\phi(\sigma_j)\subseteq\sigma_j$
  for $j\ge0$.
  \end{cor}
We thus obtain an independent proof of~\cite[Lemma~7]{eigenval};
see also~\cite{dyncomp}.
%
%
%
%
\subsection{The case $|\mu_1|=|\mu_2|$}
Write $\lambda=|\mu_1|=|\mu_2|$. 
There are two subcases.
%
%
\subsubsection{The diagonalizable case}
First consider the case when $\phi:N_\C\to N_\C$
is diagonalizable. 
When $\mu_1/\mu_2$ is not a root of unity, Favre~\cite{Favre} 
observed that $f$ cannot be made torically stable even on an 
irregular fan~\cite{Favre}. 
Indeed, the orbit $\bigcup_{n\ge0}\phi^n(\rho)$ of any
ray $\rho$ is dense in $N_\R$, so stability is impossible 
in view of Lemma~\ref{stablma}.

Now suppose $\mu_1/\mu_2$ is a root of unity.
Then $\phi^n=\lambda^n\id$ for some $n>0$,
where $\lambda=|\mu_1|=|\mu_2|$.
This implies that when $f$ is stable, $\phi$ must map any 
ray to another ray in the fan. 
We can only achieve this in special cases, such as when
$\mu_1=\mu_2$. Indeed, then 
$\phi=\pm\lambda\id$ and any symmetric fan is invariant. 

The following example illustrates the problems
that may appear when $\mu_1/\mu_2$ is a root of unity different from $1$. 
See also~\cite[Exemple~2]{Favre}.
\begin{ex}\label{motex2}
Let $\phi:N\to N$ be given by the matrix
  $A=A_\phi=\bigl( \begin{smallmatrix} 
    -1&-1\\ 3&-1 
  \end{smallmatrix} \bigr)$. Then $\mu_j=2e^{2\pi ij/3}$, $j=1,2$. In particular, $\phi^3=8\id$.  We claim that no complete regular fan $\Delta$ can be invariant by $\phi$. 
 To see this, consider any ray in $\Delta$ and
  let $v\in N$ be the corresponding primitive vector.
  Then $\phi(v)=lv'$ where $v'$ is another primitive vector and
  $l=l(v)\in\N$. If $v_1$ and $v_2$ are the primitive vectors of two adjacent rays in $\Delta$,
  then $l(v_1)l(v_2)=|\det\phi|=4$, since $\Delta$ is regular.
  Thus there are two cases: either $l(v)=2$ for all $v$,
  or $\{l(v_1),l(v_2)\}=\{1,4\}$ for any two adjacent 
  primitive vectors $v_1, v_2$. The first case is not possible
  as all entries in $\phi$ would have to be even. The second
  case cannot occur in view of $\phi^3=8\id$.
\end{ex}  
%
%
\subsubsection{The non-diagonalizable case}\label{S102}
Finally assume $\phi:N_\C\to N_\C$ is not diagonalizable.
Then $\mu_1=\mu_2=\pm\lambda$, where $\lambda\in\N$.
There exists a primitive lattice point 
$v\in N$ such that $\R v$ is 
the eigenspace for $\phi:N_\R\to N_\R$.
After subdividing, we may assume $\Delta$ is regular,
symmetric and that $\sigma\=\R_+v$ and $-\sigma$
are cones in $\Delta$. 
Pick $w\in N$ such that 
$(v,w)$ are generators for $N$ and 
such that $\R_+w$ and $-\R_+w$ are cones in $\Delta$.
The matrix of $\phi$ is given by
$A=A_\phi=\bigl( \begin{smallmatrix} 
  a&b\\ 0&a 
\end{smallmatrix} \bigr)$
where $a=\pm\lambda$ and $b\in\Z$, $b\ne0$.
Replacing $w$ by $-w$ if necessary, we have $b>0$.

First assume $a=\lambda$.
Let $\tau\in\Delta(2)$ be the unique
cone contained in $\R_+v+\R_+w$ and having
$\sigma$ as one of its faces. Then 
$\phi(\tau)\subsetneq\tau$ and 
$\phi(-\tau)\subsetneq-\tau$.
We can now proceed as in the proof of Lemma~\ref{lma:criterion1}
and refine $\Delta$ into a symmetric 
regular fan $\Delta'$ such that $\pm\tau\in\Delta'$
and for all rays $\sigma'\in\Delta(1)$
and all $n\ge0$ we have either $\phi^n(\sigma')\in \Delta(1)$
or $\phi^n(\sigma')\subseteq\pm\tau$. Then $\phi$ is torically stable
on $\Delta'$. The case when $a=-\lambda$ is handled in the same 
way, keeping all fans symmetric.
%
%
%
%
\subsection{Proof of Theorems~C and~C'}
We now have all ingredients necessary to complete the proof of Theorem~C.
In case~(a), that is, $|\mu_2|<1$, we are done by
Proposition~\ref{P101}.
In case~(b), that is, $|\mu_1|>|\mu_2|$ and $\mu_1,\mu_2\in\Z$,
the result follows as explained in Section~\ref{S101}.

Finally consider case~(c), that is, $\mu_1,\mu_2\in\R$ and $\mu_1\mu_2>0$.
If $\mu_1$ and $\mu_2$ are irrational, then we are done by 
Proposition~\ref{prop:dim2stab1}, so having treated cases~(a) and~(b), 
we may assume $\mu_1=\mu_2\in\Z$. Then either $\phi=\mu_1\id$,
with the theorem being trivial, or $\phi$ is not diagonalizable over $\C$,
in which case the theorem follows from the discussion in Section~\ref{S102}.

In fact, we have also proved Theorem~C', except for the 
case when $\mu_1$, $\mu_2$ are real, irrational, and of different sign.
We can then refine the original fan $\Delta$ so that it contains 
(possibly irregular) cones $\sigma_1$, $\sigma_2$ 
for which $\phi(\sigma_1)\subseteq\pm\sigma_1$ and 
$\phi^{-1}(\sigma_2)\subseteq\pm\sigma_2$. The proof of 
Lemma~\ref{lma:criterion1} now goes through and produces
a refinement $\Delta'$ of $\Delta$ on which $\phi$ is 
torically stable. In fact, the only irregular cones
in $\Delta'$ are $\pm\sigma_1$ and $\pm\sigma_2$.
%
%
%
%
%

\section{Stabilization - Proof of Theorems A and A'}\label{sec:proofA}
Throughout this section we assume that $\phi: N\to N$ has distinct 
and positive eigenvalues.
To prove Theorems A and A' we will use the criterion in 
Corollary~\ref{invkoner}.

The mapping $\phi:N_\R\to N_\R$ induces a mapping $\phi^*:M_\R\to M_\R$, defined by $\langle \phi^*\xi, v\rangle := \langle \xi, \phi (v)\rangle$ and with the same eigenvalues as $\phi$. Given a one-dimensional eigenspace $E\subseteq N_\R$ of $\phi$, let $\widetilde E\subseteq M_\R$ denote the corresponding eigenspace of $\phi^*$, and let $E^\perp:=\{v\in N_\R \mid \langle \xi, v\rangle =0~~\forall \xi\in \widetilde E\}\subseteq N_\R$. 
Note that, since the eigenvalues of $\phi$ are distinct,
$E^\perp$ is spanned by the eigenvectors that are not in $E$. 

\subsection{Real dynamics}\label{step1}
We say that a set $Z\subseteq N_\R$ is 
\emph{invariant} (under $\phi$) if $\phi (Z)\subseteq Z$. 

The following result is well-known, see for example \cite[Exercise~13, p.552]{Lang}.

\begin{lma}\label{tuesday}
  Any invariant subspace $V\subseteq N_\R$
  is spanned by eigenvectors of $\phi$. 
\end{lma}

Given an invariant subspace $V\subseteq N_\R$, let $E_1,\dots, E_{\dim V}$ be the invariant eigenspaces of $\phi$, corresponding to the eigenvalues $\nu_1 >\dots > \nu_{\dim V}>0$, that span $V$.  For $1\leq j\leq \dim V$, let $V_j:=E_j\oplus\dots\oplus E_{\dim V}$. Then we have a filtration $V=V_1\supsetneq V_2\supsetneq \dots \supsetneq V_{\dim V+1}:=\{0\}$ and if $v\in V_j\setminus V_{j+1}$, then $\phi^n(v)\to E_j$ when $n\to\infty$.

\subsection{Invariant rational subspaces}\label{step2}
We say that a subspace $V\subseteq N_\R$ is \emph{rational} if $\overline{V\cap N_\Q}=V$. This is equivalent to the lattice $N\cap V$ having rank equal to $\dim V$. 
A subspace $V\subseteq N_\R$ is rational if and only if its 
annihilator $V^o:=\{\xi\in M_\R \mid \xi|_V\equiv 0\}\subseteq M_\R$ is rational. Note that $(V^o)^o=V$.

Assume that $V$ and $W$ are rational subspaces. Then $V+W$ is rational and hence so is $V\cap W=(V^o+W^o)^o$. Given $V\subseteq N_\R$ it follows that there is a minimal rational subspace of $N_\R$ that contains $V$ and a maximal rational subspace contained in $V$. 
\begin{lma}\label{rummena}
Assume that $V\subseteq N_\R$ is invariant under $\phi$. Then the minimal rational subspace that contains $V$ and the maximal rational subspace contained in $V$ are both invariant.
\end{lma}
\begin{proof}
  Let $W$ be the minimal rational subspace that contains $V$. 
  Then $V\subseteq W\cap \phi (W)=W$, since $W$ is minimal. 
  To conclude, $\phi(W)=W$. The second statement follows from the first, 
  using annihilators.
\end{proof}

The mapping $\phi$ induces a binary tree $T(\phi)$ of rational invariant subspaces of $N_\R$, which should be compared to the real filtration in Section~\ref{step1}.
The nodes of $T(\phi)$ are of the form $(V,W)$, where $V$ and $W$ are rational invariant subspaces of $N_\R$, such that $V\subseteq W$. The root of $T(\phi)$ is $(\{0\},N_\R)$ and $(V,W)$ is a leaf if $V=W$. 
Assume that $V\neq W$. 
Among all one-dimensional eigenspaces $E$ of $V$ such that $E\subseteq W$, but $E\not\subseteq V$, let $E(V,W)$ be the one with the largest eigenvalue. 
Let $V'$ be the smallest rational subspace that contains $V+E(V,W)$ and let $W'$ be the largest rational subspace contained in $W\cap E(V,W)^\perp$. 
Then the two children of $(V,W)$ are $(V',W)$ and $(V,W')$. Note that $V'\subseteq W$ since $V$ and $E(V,W)$ are contained in $W$ and $W$ is rational, and that $V\subseteq W'$ since $W$ and $E(V,W)^\perp$ contain $V$ and $V$ is rational. Observe, in light of Lemma~\ref{rummena}, that $V'$ and $W'$ are invariant. 

\begin{lma}\label{dikotomi}
Let $(V,W)$ be a node in $T(\phi)$ and $U$ a rational invariant subspace such that $V\subseteq U\subseteq W$. Then either $E(V,W)\subseteq U$ or $U\subseteq E(V,W)^\perp$.
\end{lma}

\begin{proof}
Pick $x\in M_\R$, such that $E^\perp=\{x=0\}$, where $E:=E(V,W)$. Assume $U\not\subseteq E^\perp$, and pick $v \in U$, such that $x(v)\neq 0$. Then $v\notin E^\perp \supseteq V$. Let $\widetilde W=W/V$, let $\tilde\phi:\widetilde W\to\widetilde W$ be the map induced by $\phi$, and let $\widetilde E$, $\widetilde U$, and $\widetilde v$ be the images of $E$, $U$, and $v$, respectively, under the quotient map $W\to \widetilde W$. 
Then $\widetilde E$ is an eigenspace for $\tilde \phi$ with eigenvalue $\nu$ dominating all other eigenvalues of $\tilde\phi$. Thus $\nu^{-n}\tilde\phi^n(\tilde v)$ converges to a nonzero element of $\widetilde E$. 
This implies that $\widetilde E\subseteq \widetilde U$, since $\tilde v\in\widetilde U$ and $\widetilde U$ is invariant under $\tilde\phi$. It follows that $E\subseteq U$. 
\end{proof}

Let us create a new tree from $T(\phi)$. Replace each node $(V,W)$ in $T(\phi)$ by $V$ and thereafter collapse all edges between nodes $V$ and $V$. We will refer to the tree so obtained as the \emph{reduced tree} induced by $\phi$ and denote it by $\tred(\phi)$.  Observe that the nodes in $\tred(\phi)$ are in one-to-one correspondence with the leaves in $T(\phi)$. Given a node $V$ in $\tred (\phi)$ with parent $V'$, among all one-dimensional eigenspaces of $\phi$ in $V\setminus V'$, let $E(V)$ be the one corresponding to the largest eigenvalue. Then, by construction, $V$ is the smallest (invariant) rational subspace of $N_\R$ that contains $V'+E(V)$. 

We claim that all rational invariant subspaces of $N_\R$ are in $\tred(\phi)$. To see this, given a rational invariant subspace $U$, let $S(U)=\{(V,W)\in T(\phi) \mid V\subseteq U\subseteq W\}$. 
Note that $S(U)$ is non-empty, since $(\{0\},N_\R)\in S(U)$. Pick $(V,W)\in S(U)$. By Lemma~\ref{dikotomi} either $E:=E(V,W)\subseteq U$ or $U\subseteq E^\perp$. In the first case $V'\subseteq U$, where $V'$ is the smallest rational invariant subspace of $N_\R$ that contains $V+E$. In the second case $U\subseteq W'$, where $W'$ is the largest rational invariant subspace that is contained in $W\cap E^\perp$. Thus, exactly one of the children of $(V,W)$ is in $S(U)$. It follows that $S(U)$ is a maximal chain in $T(\phi)$. In particular $S(U)$ contains a leaf of  $T(\phi)$, which has to be of the form $(U,U)$. Hence $U$ is a node in $\tred(\phi)$. It is, however, not true that $(V,W)$ is a node in $T(\phi)$ as soon as $V\subseteq W$ are rational and invariant.

\subsection{Invariant chambers}\label{chambers}

With each node $(V,W)$ in $T(\phi)$ we associate a \emph{chamber} $C(V,W)$. The chamber $C(V,W)$ is an invariant open dense subset of $W$ and is defined recursively as follows. 
First let $C(\{0\},N_\R)=N_\R$.
Then, if $C(V,W)$ is defined and $(V',W)$, $(V,W')$ 
are the children of $(V,W)$, 
let $C(V',W):=C(V,W)\setminus E(V,W)^\perp$ and $C(V,W'):=C(V,W)\cap W'$. Note that $C(V,W)\cap N_\Q$ is a disjoint union of $C(V',W)\cap N_\Q$ and $C(V,W')\cap N_\Q$.  In particular, the chambers associated with the leaves of $T(\phi)$ induce invariant partitions of $N_\Q$ and $N$ (but not of $N_\R$, in general).

To the node $V$ in $\tred(\phi)$ associate the chamber $C(V):=C(V,V)$. Then the chambers $C(V)$ provide partitions of $N_\Q$ and $N$. More precisely, given a node $V'$ in $\tred(\phi)$, the chambers $C(V)$, where $V$ ranges over ancestors of $V'$ in $\tred (\phi)$, give partitions of $V'\cap N_\Q$ and $V'\cap N$. Assume that the genealogy of $V$ is the chain of nodes in $\tred(\phi)$:  
\begin{equation}\label{gen}
\{0\}=V_0\subsetneq V_1\subsetneq \dots \subsetneq V_s=V. 
\end{equation}
For $1\leq k\leq s$ pick $x_k\in M_\R$, such that $E(V_k)^\perp=\{x_k=0\}$. Then 
\begin{equation*}
C(V)=V\setminus \bigcup_{k=1}^s E(V_k)^\perp = 
V\cap \bigcap_{k=1}^s\{x_k\neq 0\}. 
\end{equation*}
Observe that $V$ is the smallest rational subspace of $N_\R$ that contains the subspaces $\{E(V_k)\}_{1\leq k\leq s}$.

\begin{lma}\label{newadapted}
Let $V$ be a node in $\tred(\phi)$ and let $v\in C(V)$. Then there is no rational invariant proper subspace $U\subsetneq V$ containing $v$.
\end{lma}
\begin{proof}
Let \eqref{gen} be the genealogy of $V$ in $\tred(\phi)$, with corresponding $x_k\in M_\R$, and let $U$ be the smallest rational invariant subspace of $V$ containing $v$. Pick $r$ maximal such that $V_r\subseteq U$. Assume $r<s$. Since $v\in C(V)$,  $x_{r+1}(v)\neq 0$. By arguments as in the proof of Lemma~\ref{dikotomi} one can show that $E(V_{r+1})\subseteq U$. 
Since $V_{r+1}$ is the smallest rational invariant subspace of $V$ that contains $V_r+E(V_{r+1})$, $V_{r+1}\subseteq U$, which contradicts the maximality of $r$. Hence $U=V$, which proves the lemma. 
\end{proof}

The chamber $C(V)$ admits a further decomposition into $2^s$ connected components. Given $x_1,\dots, x_s\in M_\R$ and $\eta=(\eta_1,\dots, \eta_s)\in\{\pm 1\}^{s}$, let $C(V,\eta):=V\cap \bigcap_{j=1}^s\{\eta_jx_j>0\}$; we will refer to $\eta$ as a \emph{sign vector}. Then the $C(V,\eta)$ are clearly disjoint and $C(V)=\bigcup_{\eta\in\{\pm 1\}^s}C(V,\eta)$. Hence the chamber components $C(V,\eta)$, where $V$ ranges over the nodes in $\tred (\phi)$ and $\eta$ over possible sign vectors, provide partitions of $N$ and $N_\Q$. Moreover, if the eigenvalues of $\phi$ are positive, then each $C(V,\eta)$ is invariant under $\phi$. 
If $V'$ is an ancestor of $V$, say $V'=V_{s'}$, we will refer to $\eta':=(\eta_1,\dots, \eta_{s'})$
as the \emph{truncation} of $\eta=(\eta_1,\dots, \eta_{s'},\eta_{s'+1},\dots, \eta_s)$.

For each $(V,\eta)$, let $E(V,\eta):=E(V)\cap C(V,\eta)$ and let $e(V,\eta)\in N_\R$ be a generator for the ray $E(V,\eta)$. 

\subsection{Adapted system of cones}\label{adapt}
We define an \emph{adapted system of cones} to be a collection of simplicial cones $\sigma(V,\eta)$, where $V$ runs over the vertices in $\tred(\phi)$ and $\eta$ over possible sign vectors, that satisfies the following conditions
\begin{enumerate}
\item[(A1)]
$\interior \sigma(V,\eta)\subseteq C(V,\eta)$ and $\sigma(V,\eta)$ spans $V$
\item[(A2)]
if $V'\subseteq V$ is the parent of $V$ in $\tred(\phi)$, and $\pi:V\to V/V'$ is the natural projection, then $\sigma(V,\eta)\cap V'=\sigma(V',\eta')$, where $\eta'$ is the truncation of $\eta$, and $\pi(e(V,\eta))\in \interior \pi(\sigma(V,\eta))$.
\end{enumerate}
We say that the system is \emph{rational} if all cones $\sigma(V,\eta)$ are rational, and \emph{invariant} (under $\phi$) if each cone is invariant (under $\phi$).

\begin{lma}\label{attraherar}
Let $\mathcal S=\{\sigma(V,\eta)\}$ be an adapted system of cones and $v\in N$.  Then there exists $n_0=n_0(v)\in\N$ such that for $n\geq n_0$, $\phi^n(v)\in \sigma(V,\eta)$ for some $\sigma(V,\eta)\in\mathcal S$. More precisely, if $v\in C(V,\eta)$, then $\phi^n(v)\in\sigma(V,\eta)$ for $n\geq n_0$. 
\end{lma}

\begin{proof}
From Section~\ref{chambers} we know that $v$ is contained in a unique chamber $C(V,\eta)$. Let $\sigma=\sigma(V,\eta)$ be the corresponding cone in $\mathcal S$. Assume that \eqref{gen} 
is the genealogy of $V$ in $\tred(\phi)$. Write $e_k:=e(V_k,\eta^k)$ and $\sigma_k:=\sigma(V_k,\eta^k)$, where $\eta^k$ is the truncation of $\eta$. 
Then for $1\leq k\leq s$ the cone $\sigma_k\in \mathcal S$ is of the form 
$\sigma_k=\sigma_{k-1}+\sum_{j=1}^{m_k}\R_+v_{k,j}$, where $\sigma_0=\{0\}$, $m_k:=\dim V_k -\dim V_{k-1}$,  and $v_{k,j}\in C(V_k,\eta^k)$, so that $\sigma=\sigma_s=\sum_{k=1}^s\sum_{j=1}^{m_k}\R_+v_{k,j}$. Moreover, $\pi_k(e_k)\in\interior \pi_k(\sum_{j=1}^{m_k}\R_+ v_{k,j})$, where $\pi_k:V_k\to V_k/V_{k-1}$ is the natural projection. 

For $1\leq k\leq s$, choose $x_k\in V^*$ such that $\langle x_k,e_j\rangle=\delta_{kj}$.
We 
identify $(V/V_{k-1})^*$ with $\{\xi\in V^*\mid \xi|_{V_{k-1}}=0\}$. Then $\phi^*$ induces a self-mapping on $(V/V_{k-1})^*$ and if $\langle \xi, e_k\rangle>0$, then $\phi^{*n}\xi\to\R_+ x_k$ when $n\to\infty$. Indeed, the subspace $\R_+x_k\subseteq (V/V_{k-1})^*$ is the one with the largest eigenvalue. 

The dual cone $\sigma^*$ of $\sigma$ is of the form $\sigma^*=\sum_{k=1}^{s}\sum_{j=1}^{m_k}\R_+ \xi_{k,j}$, where $\xi_{k,j}\in\ker (V^*\to V_{k-1}^*)\cong (V/V_{k-1})^*$ and $\langle \xi_{k,j}, e_k \rangle >0$; in particular, $\phi^{*n}\xi_{k,j}\to \R_+x_k$. 

Since $v\in C(V,\eta)$, $\langle x_k,v\rangle >0$ for $1\leq k \leq s$. By continuity, there is an $n_0=n_0(v)\in \N$ such that $\langle \xi_{k,j},\phi^n(v)\rangle=\langle \phi^{*n}\xi_{k,j}, v\rangle>0$ for $1\leq k\leq s$, $1\leq j\leq m_k$ and $n\geq n_0$. Thus $\phi^n(v)\in\interior \sigma$ for $n\geq n_0$. 
\end{proof}

\begin{remark}\label{attrmk}
As can be seen from the proof, the first part of Lemma~\ref{attraherar} 
remains valid if some of the eigenvalues $(\mu_i)_{i=1}^m$
of $\phi$ are negative as long as $|\mu_1|>\dots > |\mu_m|>0$. In general, if $v\in C(V)$, then, for $n\geq n_0$,  $\phi^n(v)\in\sigma(V,\eta)$ for some sign vector $\eta$. 
\end{remark}

\begin{lma}\label{invariant}
Let $\mathcal S= \{\sigma(V,\eta)\}$ be an adapted system of cones. Then there exists $n_0\in\N$, such that $\mathcal S$ is invariant under $\phi^n$ for $n\geq n_0$. 
\end{lma}

\begin{proof}
To each node $V$ in $\tred(\phi)$,  we will associate $n_0(V)\in\N$, 
such that $\phi^n(\sigma(V,\eta))\subseteq\sigma(V,\eta)$ for all sign vectors $\eta$ and $n\geq n_0(V)$; this is done by induction over $\tred(\phi)$. 

Set $n_0(\{0\})=0$. Let $V$ be a node in $\tred (\phi)$, such that $n_0(V')$ is defined, where $V'$ is the parent of $V$. Pick a sign vector $\eta$ and let $\eta'$ be the truncation. Then $\sigma(V,\eta)$ is of the form $\sigma(V,\eta)=\sigma(V',\eta')+\sum_{j=1}^{m'}\R_+v_j$ for some $v_j\in C(V, \eta)$ and $m'=\dim V-\dim V'$. From Lemma~\ref{attraherar} we know that for $1\leq j\leq m'$, there is a $n_0(v_j)\in \N$, such that $\phi^n(v_j)\in\interior\sigma(V,\eta)$ for $n\geq n_0(v_j)$. Let $n_0(V,\eta):=\max_{1\leq j\leq m'} n_0(v_j)$, and let $n_0(V):=\max(n_0(V'), \max_\eta n_0(V,\eta))$. 

Finally, set $n_0:=\max_{V\in\tred(\phi)}n_0(V)$. Then $n_0$ has the desired properties. 
\end{proof}

\begin{remark}\label{lol}
Following the proof of Lemma~\ref{invariant}, one can prove that if the eigenvalues of $\phi$  satisfy $|\mu_1|>\dots>|\mu_m|>0$, then there exists $n_0\in \N$, such that, for $n\geq n_0$, $\phi^n$ maps each $\sigma(V,\eta)$ into $\sigma(V,\eta')$ for some sign vector $\eta'$.
\end{remark}

The idea of the proofs of Theorems A and A' is to refine $\Delta$ so that it contains an invariant (under $\phi$ and $\phi^n$, respectively) adapted system of rational cones. 
Then the results follow by applying Corollary~\ref{invkoner}. 
First we need a few preliminary results on adapted systems of cones.

\begin{lma}\label{projicering}
Let $\Delta$ be a fan in $N$ and write $V=N_\R$. Given $\sigma'\in\Delta$, let $V'=\spann \sigma'$, and let $\pi:V\to V/V'$ be the natural projection. Then for each $v\in V/V'$, there exists at most one $\sigma$, such that $\sigma\supseteq \sigma'$ and $v\in\interior \pi(\sigma)$. If $\Delta$ is complete, there is a unique such $\sigma$.
\end{lma}
\begin{proof}
Let $\stjarna (\sigma'):=\{\sigma\in\Delta \mid \sigma \supseteq \sigma'\}$. Then $\Delta_{\sigma'}:=\{\pi(\sigma) \mid \sigma\in\stjarna ({\sigma'})\}$ is a fan in $N/N'$, where $N'$ is the sublattice of $N$ generated by ${\sigma'}\cap N$, see~\cite[Section~3.1]{F1}. If $\Delta$ is complete, then so is $\Delta_{{\sigma'}}$. Moreover, there is a one-to-one correspondence between the cones in $\stjarna({\sigma'})$ and $\Delta_{\sigma'}$. In particular there is a most one cone in $\stjarna (\sigma')$, such that $\pi(\sigma)$ contains $v$, and if $\Delta$ is complete there exists such a $\sigma$. 
\end{proof}

\begin{lma}\label{maxett}
  Any fan $\Delta$ admits at most one rational adapted system of cones.
\end{lma}
\begin{proof}
Let $V$ be a node in $\tred (\phi)$ and let $\eta$ be a sign vector. Note that the collection of cones in $\Delta$ that are contained in $V$ form a fan. Suppose that $\sigma'\in\Delta$ spans $V'\subseteq V$. Then, by Lemma~\ref{projicering}, there is at most one cone $\sigma\in\Delta$, such that $\sigma'\subseteq\sigma\subseteq V$  and such that $\interior\pi(\sigma)\ni\pi(e(V,\eta))$, where $\pi$ is the projection $\pi:V\to V/V'$. Thus there is at most one cone $\sigma(V,\eta)$ satisfying (A1) and (A2).
\end{proof}

\begin{lma}\label{forfinas}
Let $\Delta$ be a fan in $N$. Assume that $\Delta$
contains an adapted system of cones and that 
for every invariant rational subspace $V$ of $N_\R$, there is a subfan of $\Delta$ whose support equals $V$. Then every refinement of $\Delta$ 
contains a unique adapted system of cones.
\end{lma}
\begin{proof}
Uniqueness follows from Lemma~\ref{maxett}. 
Assume that $\Delta$ satisfies the assumption of the lemma and let $\mathcal S=\{\sigma(V,\eta)\}$ denote the adapted system of cones. Let $\Delta'$ be a refinement of $\Delta$. 

We will inductively find cones $\tau(V,\eta)\in\Delta'$ that satisfy (A1) and (A2). 
Let $\tau(\{0\},\eta):=\{0\}$. Let $V$ be a node in $\tred (\phi)$, with parent $V'$, and $\eta$ a sign vector. Assume that we have found $\tau'=\tau(V',\eta')$, where $\eta'$ is the truncation of $\eta$. Note that $\sigma(V',\eta')$ then is the smallest cone in $\Delta$ that contains $\tau'$. Let $\pi:V\to V/V'$ be the natural projection. Since there is a subfan of $\Delta'$ with support $V'$, Lemma~\ref{projicering} asserts that there is a unique cone $\tau\subseteq V$, that contains $\tau'$ and satisfies that $\pi(e(V,\eta))\in\interior \pi(\tau)$. In particular, there is $v\in \tau$, such that $\pi (v)=\pi (e(V,\eta))$, that is, $v=e(V,\eta)+v'$ for some $v'\in V'$. Thus $\spann \tau$ contains $E(V)$ and since it also contains $V'$ and is rational, it follows that $\spann \tau$ contains $V$. Hence $\tau$ spans $V$. 

It remains to show $\interior \tau\subseteq C(V,\eta)$. Let $\tilde\sigma$ be the smallest cone in $\Delta$ that contains $\tau$. Then $\tilde\sigma$ spans $V$, since there is a subfan of $\Delta$ with support $V$, and $\pi(e(V,\eta))\in\interior \pi(\tilde\sigma)$. Moreover, since there is a subfan of $\Delta$ with support $V'$, $\tilde\sigma\cap V'$ is a face of $\tilde\sigma$ and thus $\tilde\sigma\cap V'\in\Delta$. On the other hand $\tilde\sigma \supseteq \tau\supseteq \tau'$ so that $\tilde\sigma\cap V'\supseteq \tau'$. Since $\tau'$ has maximal dimension in $V'$, $\tilde \sigma \cap V'=\sigma(V',\eta')$ and thus $\tilde\sigma\supseteq \sigma(V',\eta')$. Now Lemma~\ref{projicering} implies that, in fact, $\tilde\sigma=\sigma(V,\eta)$. Hence $\interior\tau(V,\eta)\subseteq\interior \sigma (V,\eta)\subseteq C(V,\eta)$. To conclude, $\tau(\sigma, V)$ satisfies (A1) and (A2).
\end{proof}

\begin{lma}\label{rationella}
There exists a rational adapted system of cones. 
\end{lma}
\begin{proof}
We will construct rational cones $\sigma(V,\eta)$ inductively. First let $\sigma(\{0\},\eta)=\{0\}$. Now let $V\ne\{0\}$ be a node in $\tred(\phi)$ and $\eta$ a sign vector. Assume that $\sigma(V',\eta')$ is constructed, where $V'$ is the parent of $V$ in $\tred(\phi)$ and $\eta'$ is the truncation of $\eta$. Moreover, assume that the genealogy of $V'=V_s$ is given by \eqref{gen}. Write $m':=\dim V-\dim V'$ and $V=V' + \widetilde V$, and pick $x\in M_\R$, such that $E(V)^\perp=\{x=0\}$ and $x(e(V,\eta))>0$. For $1\leq i\leq m'$, pick $\tilde s_i\in \widetilde V$, such that $e(V,\eta)\in\interior \sum_{i=1}^{m'}\R_+ \tilde s_i$ and $x(\tilde s_i)>0$. Next, let $v_i$ be a rational perturbation of 
\begin{equation}\label{detroit}
e(V_1,\eta^1)+\dots + e(V_{s-1}, \eta^{s-1})+e(V_s,\eta^s)+ \tilde s_i,
\end{equation}
where $\eta^k$ are truncations of $\eta=\eta^s$. 
Since $V$ is rational, we can find arbitrarily small such perturbations. 
Note that, provided the perturbation is small enough, $v_i\in C(V,\eta)$. Finally, let $\sigma(V,\eta):=\sigma(V',\eta')+\sum_{i=1}^{m}\R_+ v_i$. If the perturbations $v_i$ of \eqref{detroit} are small enough, then $\sigma(V,\eta)$ satisfies properties (A1) and (A2). 
\end{proof}

\subsection{Invariant adapted systems of real cones}\label{step3}
In this section we will construct a real (not necessarily rational) invariant adapted system of cones $\mathcal G=\{\Gamma(V,\eta)\}$. 
Later, in Section~\ref{step6}, we will perturb the cones in $\mathcal G$ into rational cones. 

Let \eqref{gen} be the genealogy of $V$ in $\tred(\phi)$ and pick a sign vector $\eta\in\{\pm 1\}^s$. 
For $1\leq k \leq s$, let $\eta^k$ be the truncation of $\eta$, $e_k:=e(V_k,\eta^k)$ with corresponding eigenvalue $\nu_k$, and $m_k:=\dim V_k-\dim V_{k-1}$. Moreover, choose nonzero eigenvectors $e_{k,i}$, ordered so that $\nu_{k,1}>\dots >\nu_{k,m_k}$ and $e_k=e_{k,1}$. Then $V_k=V_{k-1}\oplus \widetilde V_k$, where $\widetilde V_k = \bigoplus_{i=1}^{m_k}\R e_{k,i}$. 

Given parameters $\delta_1, \delta_2, \dots, \delta_s >0$ and $\varepsilon_2,\dots, \varepsilon_s >0$, let $\gamma_1:=1$ and $\gamma_k:=\varepsilon_2\dots\varepsilon_k$ for $k\geq 2$. Further, for $1\leq k\leq s$ and $1\leq i\leq m_k$ set 
\begin{align} 
 v_{1,i} &= e_1+\delta_1\tilde v_{1,i}\\
  v_{k,i} &= e_1+2^{-1}\gamma_2e_2+\dots+2^{2-k}\gamma_{k-1}e_{k-1}
  +2^{2-k}\gamma_k(e_k+\delta_k\tilde v_{k,i})\quad\text{if $k>1$}
\end{align}
where
\begin{equation*}
\tilde v_{k,i}=
\left \{
\begin{array}{ll}
e_{k,2}+\dots +e_{k,i}- e_{k,i+1} & \text{ if } 1\leq i < m_k\\
e_{k,2}+\dots +e_{k,m_k} & \text{ if } i = m_k
\end{array}\right.
\end{equation*}
Here $\tilde v_{k,1}$ should be interpreted to be equal to $-e_{k,2}$ if $m_k\geq 2$ and $0$ otherwise. Note that $v_{k,i}\in V_k$ and $\tilde v_{k,i}\in \widetilde V_k$. Also note that $v_{k,i}$ and $\tilde v_{k,i}$ depend on the sign vector $\eta$, since $e_{k,i}$ do.   Finally note that 
$\sum_{i=1}^{m_k}\R_+ (  e_k + \delta_k \tilde v_{k,i})$ is a simplicial real cone in $\widetilde V_k$ of dimension $m_k=\dim \widetilde V_k$ containing $ e_k$ in its interior. 

Now let 
\[
\Gamma(V_k,\eta^k):=
\sum_{j=1}^k\sum_{i=1}^{m_j}\R_+ v_{j,i}=
\Gamma(V_{k-1}, \eta^{k-1}) + \sum_{i=1}^{m_k}\R_+ v_{k,i}\subseteq V_k.
\]
Observe that $\Gamma(V_k, \eta^k)\cap V_{k-1}=\Gamma(V_{k-1},\eta^{k-1})$, since the coefficients of $e_k$ in $v_{k,i}(\eta)$ are positive for $1\leq i \leq m_k$. 

To show that $\Gamma(V,\eta)=\Gamma(V_s,\eta^s)$ satisfies properties (A1)-(A2), let us give a dual description of $\Gamma(V,\eta)$ in $V=V_s$. Let 
$\{x_{\ell,j}\}_{1\leq \ell\leq s, 1\leq j\leq m_\ell}$ be the basis of $V^*$ dual to $\{e_{k,i}\}_{1\leq k\leq s, 1\leq i\leq m_k}$, so that 
 $\langle x_{\ell,j}, e_{k,i}\rangle=1$ if $\ell=k$ and $j=i$ and $\langle x_{\ell,j}, e_{k,i}\rangle=0$ otherwise. Write $x_\ell:=x_{\ell,1}$. 

For $1\leq \ell < j\leq s$, let $a_{\ell,j}=\varepsilon_{\ell+1}^{-1}\dots\varepsilon_j^{-1}$, and let 
\[
\xi_{\ell,j}:=\delta_\ell^{-1}\tilde\xi_{\ell,j}+x_\ell-(a_{\ell,\ell+1}x_{\ell+1}+\dots +a_{\ell,s} x_s),
\]
where
\begin{equation*}
\tilde \xi_{\ell,j}:=
\left\{
\begin{array}{ll}
x_{\ell,2}+2 x_{\ell,3}+\dots + 2^{j-2} x_{\ell,j}- 2^{j-1}x_{\ell,j+1} & \text{ if } 1\leq j < m_\ell\\
x_{\ell,2}+ 2 x_{\ell,3}+\dots +2^{m_\ell-2} x_{\ell,m_\ell} & \text{ if } j = m_\ell
\end{array} \right.
\end{equation*}
Here $\tilde \xi_{\ell, 1}$ should be interpreted as $-x_{\ell,2}$ if $m_\ell\geq 2$ and $0$ otherwise. 

A computation yields that 
$\langle \xi_{\ell, j}, v_{k,i}\rangle>0$ if $\ell=k$ and $i=j$ and $\langle \xi_{\ell, j}, v_{k,i}\rangle=0$ otherwise, so that the dual cone 
$\Gamma(V,\eta)^*=\sum_{\ell=1}^s\sum_{j=1}^{m_\ell}\R_+ \xi_{\ell,j}(\eta)$. 

We claim that $\phi$ maps the open rays $\R^*_+v_{s,1},\dots, \R^*_+ v_{s,m_s}$ into $\interior\Gamma(V,\eta)$.

To prove the claim, observe first that
\begin{equation}\label{avbildning}
\langle \xi_{s,j}, \phi(v_{s,i})\rangle = 
2^{2-s}\gamma_s(\nu_s +\langle \tilde\xi_{s,j}, \phi(\tilde v_{k,i})\rangle ), 
\end{equation}
where
\begin{equation*}
\langle \tilde\xi_{s, j}, \phi(\tilde v_{s,i})\rangle=
\left \{
\begin{array}{ll}
\nu_{s,2}+\dots + 2^{m_s-2}\nu_{s,m_s} & \text{ if } i=j=m_s\\
\nu_{s,2}+\dots + 2^{i-1}\nu_{s,i+1} & \text{ if } i=j<m_s\\
\nu_{s,2}+\dots + 2^{I-2}\nu_{s,I}-2^{I-1}\nu_{s,I+1} & \text{ if } i\neq j; \text{ here }I=\min(i,j)
\end{array}\right.
\end{equation*}
Here the second line should be interpreted as $0$ if $i=j=1$. 
Now, the right hand side of \eqref{avbildning} is strictly positive, since $\nu_{s}>\nu_{s,2}>\dots >\nu_{s,m_s}$. 
Moreover, for $\ell< s$,
\begin{equation*}
\langle \xi_{\ell,j}, \phi(v_{s,i})\rangle = 
2^{1-\ell}\gamma_\ell(\nu_\ell-2^{-1}\nu_{\ell+1}-\dots - 2^{\ell+1-s}\nu_{s-1}-2^{\ell+1-s}\nu_s),
\end{equation*}
which is strictly positive since $\nu_1>\dots >\nu_s$. 

To conclude, $\langle \xi_{\ell,j},\phi(v_{s,i})\rangle>0$ for $1\leq \ell\leq s$ and $1\leq j\leq m_\ell$, and thus we have proved that $\phi(\R^*_+ v_{s,i})$ lies in the interior of $\Gamma(V,\eta)$. In particular, by induction, $\Gamma(V,\eta)$ is invariant under $\phi$.

\subsection{Preparation of the fan}\label{step5}
In order to prove Theorems A and A', we first
refine $\Delta$ so that for each rational invariant 
subspace $V$, that is, each node in $\tred (\phi)$, there is a subfan of $\Delta$ whose support is $V$. In particular, $\Delta$ is complete.
This is possible to do since the rational rays are dense in $V$. 

Next, we refine $\Delta$ so that it contains an adapted system of cones. This can be done as follows. Let $\mathcal S$ be a rational adapted system of cones; its existence being guaranteed by Lemma~\ref{rationella}. Let $\Delta_{\mathcal S}$ be the fan generated by the cones in $\mathcal S$, and let $\Delta'$ be a fan that refines both $\Delta$ and $\Delta_{\mathcal S}$. Then by Lemma~\ref{forfinas}, $\Delta'$ contains an adapted system of cones. 

Finally, by Lemma~\ref{L402} we can refine $\Delta'$ so that it becomes regular and projective. The resulting fan will contain a unique adapted system of cones by Lemma~\ref{forfinas}.

\subsection{Proof of Theorem~A'}
Let $\Delta'$ be the refined fan in Section~\ref{step5} and let 
$\mathcal S=\{\sigma(V,\eta)\}$ denote the unique adapted system of cones. 
Consider $\rho\in\Delta'(1)$. According to Lemma~\ref{attraherar} and Remark~\ref{attrmk} there is $n_0(\rho)\in\N$, such that, for $n\geq n_0(\rho)$ $\interior\phi^n(\rho)\subseteq \interior \sigma(V,\eta)$ for some $\sigma(V,\eta)\in\mathcal S$. Let $n_0:=\max_{\rho\in\Delta'(1)} n_0(\rho)$. 

Moreover, according to Lemma~\ref{attraherar} and Remark~\ref{lol}, for $n\geq n_0$ (with $n_0$ possibly replaced by a larger number), $V$ a node in $\tred (\phi)$, and $\eta$ a sign vector, $\phi^n(\sigma(V,\eta))\subseteq \sigma(V,\eta')$ for some sign vector $\eta'$. Now Corollary~\ref{invkoner} and Remark~\ref{teckenrmk} assert that $f^n:X(\Delta')\dashrightarrow X(\Delta')$ is $1$-stable for $n\geq n_0$, which concludes  the proof of Theorem~A'.

\subsection{Incorporation of cones}\label{step6}
We will now prove Theorem~A. 
Given a fan $\Delta$ in $N$, replace $\Delta$ by the refined fan in Section~\ref{step5} and let $\mathcal S=\{\sigma(V,\eta)\}$ be the (unique) adapted system of cones in $\Delta$. We will construct and incorporate into $\Delta$ a rational invariant adapted system of cones $\mathcal T=\{\tau(V,\eta)\}$. This will be done inductively over the reduced tree $\tred(\phi)$. In fact, the cones in $\mathcal T$ will be perturbations of the cones in the real invariant adapted system $\mathcal G$ constructed in Section~\ref{step3}. 

Let $V$ be a node in $\tred(\phi)$, with genealogy \eqref{gen}. We will construct and incorporate cones $\tau(V,\eta)$ for all possible sign vectors $\eta$ by inductively constructing and incorporating cones $\tau_k=\tau(V_k,\eta^k)$ for $1\leq k\leq s$ and all possible choices of sign vectors $\eta^k$. Let us use the notation from Section~\ref{step3}, and write $\Gamma_k:=\Gamma(V_k,\eta_k)$. 
Moreover, let $w_k:=e_1+2^{-1}\gamma_2e_2+\dots +2^{2-k}\gamma_{k-1}e_{k-1}+2^{2-k}\gamma_ke_k$ and $u_k:=e_1+\dots +2^{2-k}\gamma_{k-1}e_{k-1}+2^{1-k}\gamma_ke_k$, so that $w_k=u_{k-1}+2^{2-k}\gamma_ke_k$.

Write $\sigma_1:=\sigma(V_1,\eta^1)$. Note that $w_1=e_1\in\interior \sigma_1$. 
By continuity we can choose $\delta_1$ so that $v_{1,i}=w_1+\delta_1 \tilde v_{1,i}\in\interior \sigma_1$ for $1\leq i \leq m_1$. 
Moreover, since the rational rays are dense in $V_1$ we can find rational perturbations $t_{1,i}$ of $v_{1,i}$, such that $t_{1,i}\in\interior\sigma_1$ for $1\leq i\leq m_1$. Write $\tilde t_{1,i}=t_{1,i}-v_{1,i}$. 
Now let $\tau_1=\sum_{i=1}^{m_1} \R_+ t_{1,i}$. Then $\tau_1$ is a perturbation of $\Gamma_1$ and if $\tilde t_{1,i}$ are small enough, then $\phi(\interior \tau_1)\subseteq \interior \tau_1$, since $\phi(\interior\Gamma_1)\subseteq \interior \Gamma_1$. Also, $\tau_1$ satisfies properties (A1) and (A2) in Section~\ref{adapt} and  $u_1\in\interior \tau_1$. By Lemmas~\ref{barvinok} and~\ref{L401} we can find a simplicial and projective refinement $\Delta'$ of $\Delta$, such that $\tau_1\in\Delta'$ and all cones in $\Delta$ that do not contain $\sigma_1$ are in $\Delta'$. Replace $\Delta$ by $\Delta'$, and $\mathcal S$ by the unique adapted system of cones in $\Delta'$. Such a system exists by Lemma~\ref{forfinas}.

Write $\sigma_2:=\sigma(V_2,\eta^2)$, let $\pi_1:N_\R\to N_\R/V_1$ be the natural projection, and let $\stjarna(\tau_1):=\{\sigma\in\Delta\mid \sigma \supseteq \tau_1\}$. Since $u_1\in\interior\tau_1$, $|\stjarna(\tau_1)|$ contains a neighborhood of $u_1$ in $N_\R$. In particular, $w_2=u_1+\gamma_2e_2$ is in the interior of some cone in $\stjarna(\tau_1)$ if $\gamma_2$ is small enough and since $\pi_1(  e_2)\in\pi_1(\sigma_2)$ this cone has to be $\sigma_2$. 

By continuity, we can choose $\delta_2$ small enough so that $v_{2,i}=w_2+\delta_2 \gamma_2\tilde v_{2,i}\in\interior\sigma_2$  for $1\leq i\leq m_2$. Furthermore, we can replace $v_{2,i}$ by rational perturbations $t_{2,i}\in\interior\sigma_2$; write $\tilde t_{2,i}= v_{2,i}-t_{2,i}$. Now let $\tau_2:=\tau_1+\sum_{i=1}^{m_2} \R_+ t_{2,i}$. 
Since the rays $v_{2,i}$ are mapped into the interior of $\Gamma_2$, $\phi(t_{2,i})\in\interior \tau_2$ if $\tilde t_{2,i}$ are small enough. Hence $\tau_2$ is invariant. If $\tilde t_{2,i}$ are small enough, $\tau_2$ satisfies properties (A1) and (A2) in Section~\ref{adapt}, and $u_2=e_1+2^{-1}\gamma_2e_2\in\interior \tau_2$. 
Since $t_{2,i}\in\interior\sigma_2$, $\partial \sigma_2\cap\partial \tau_2=\tau_1$, which is a face of both $\sigma_2$ and $\tau_2$. Thus, according to Lemmas~\ref{barvinok} and~\ref{L401}, we can find a simplicial and projective refinement $\Delta'$ of $\Delta$, such that $\tau_2\in\Delta'$ and that the cones in $\Delta$ that do not contain $\sigma_2$ are in $\Delta'$. Replace $\Delta$ by such a refinement and $\mathcal S$ by the new adapted system.

Inductively assume that we have constructed and incorporated $\tau_{k-1}$ so that $u_{k-1} =  e_1+ \dots + 2^{3-k}\gamma_{k-2} e_{k-2} + 2^{2-k}\gamma_{k-1}   e_{k-1}\in\interior\tau_{k-1}$; 
here $\varepsilon_2,\dots, \varepsilon_{k-1}$, and hence $\gamma_2, \dots, \gamma_{k-1}$,  are chosen along the way. By arguments as above we can choose $\varepsilon_k$, and hence $\gamma_k$, such that $w_k=u_{k-1}+2^{2-k}\gamma_k e_k\in\interior \sigma_k$, where $\sigma_k:=\sigma(V_k,\eta^k)$. 
Moreover, we can choose $\delta_k$ and $\tilde t_{k,i}$ so that $t_{k,i}:=v_{k,i}+\tilde t_{k,i}$ are rational and contained in $\sigma_k$. 
Now, let $\tau_k:=\tau_{k-1}+\sum_{i=1}^{m_k} \R_+ t_{k,i}$. If $\tilde t_{k,i}$ are small enough $\tau_k$ is invariant and satisfies properties (A1) and (A2) in Section~\ref{adapt} and $u_k\in\interior \tau_k$. 
Since $\partial \sigma_k\cap\partial \tau_k=\tau_{k-1}$ is a face of both $\sigma_k$ and $\tau_k$, we can incorporate $\tau_k$ into $\Delta$ 
according to Lemma~\ref{barvinok} and the resulting fan will have a unique adapted system of cones. By Lemma~\ref{L401} we can choose the resulting fan projective.

We need to show that when incorporating a cone $\tau(V,\eta)\in\mathcal T$ into $\Delta$ we do not affect the cones in $\mathcal T$ already created and incorporated. 

Assume that $\hat\tau:=\tau(\widehat V, \hat\eta)$ is in $\Delta$. We claim that $\hat\tau$ is not affected when incorporating $\tau$. By Lemma~\ref{barvinok} it suffices to show that $\hat\tau$ does not contain $\sigma$. If $\widehat V=V$ but $\hat\eta=(\hat\eta_1,\dots, \hat\eta_s)\neq \eta$, then $\hat\tau\not\supseteq\sigma$, since $\interior \hat\tau\subseteq C(V, \hat\eta)$ and $\interior \sigma\subseteq C(V,\eta)$ are contained in different components of $C(V)$.

Therefore assume that $\widehat V \neq V$. Let $\{0\}=\widehat V_0\subsetneq \widehat V_1\dots \subsetneq \widehat V_r=\widehat V$ be the genealogy of $\widehat V$.  By assumption we have constructed and incorporated cones $\tau(\widehat V_k,\hat \eta^k)$ for $1\leq k\leq r$  and all possible sign vectors $\hat\eta^k$. Thus $V$ is not among the $\widehat V_k$. By construction, $\hat\tau$ is of the form $\hat\tau=\sum_{k=1}^r\sum_{i=1}^{m_k}\R_+\hat t_{k,i}$, where $\hat t_{k,i}\in C(\widehat V_k)$. From Lemma~\ref{newadapted} we know that the smallest rational invariant subspace containing $\hat t_{k,i}$ is $\widehat V_k$. Thus, the smallest rational invariant subspace of $N_\R$ containing a given face of $\hat\tau$ is among the $\widehat V_k$. Since the smallest rational invariant subspace containing $\sigma$ is $V$, we conclude that $\sigma$ is not a face of $\hat\tau$. The claim is proved.

\subsection{Incorporation of rays}\label{step7}
Let $\Delta$ be the fan in Section~\ref{step6} and let $\mathcal T$ be the rational adapted systems of cones. 
We claim that we can find a further refinement $\Delta'$ of $\Delta$ such that if $\rho$ is a ray in $\Delta'$ and $n\geq 1$, then either $\phi^n(\rho)\in\Delta'(1)$ or $\phi^n(\rho)$ is contained in a cone in $\mathcal T$. By Corollary~\ref{invkoner} $f:X(\Delta')\dashrightarrow X(\Delta')$ is then 1-stable, which proves Theorem~A. 

It remains to prove the claim. 
Since, by Lemma~\ref{attraherar}, every ray in $\Delta$ is eventually mapped into one of the cones in $\mathcal T$ it is sufficient the add to $\Delta$ the finitely many rays of the form $\phi^n(\rho)$, where $\rho\in\Delta(1)$ and $\phi^n(\rho)$ is not contained in any of the cones in~$\mathcal T$. 

Let $\rho'=\phi^n(\rho)$ be such a ray, and let $\sigma'$ be the unique cone in $\Delta$, such that $\interior\rho'\subseteq\interior\sigma'$. By Lemma~\ref{barvinok},  we can find a refinement $\Delta'$ of $\Delta$, such that $\rho'\in\Delta'(1)$, $(\Delta'(1)\setminus \rho')\subseteq \Delta(1)$, and such that if $\sigma\in\Delta$ does not contain $\sigma'$, then $\sigma\in\Delta'$. Moreover, by Lemma~\ref{L401}, $\Delta'$ can be chosen projective. Note that since $\rho'$ by assumption is not contained in any cone in $\mathcal T$, $\sigma'$ cannot be a face of a cone in $\mathcal T$. Thus all cones in $\mathcal T$ are in $\Delta'$. 

This proves the claim and thus concludes the proof of Theorem~A.

\section{Proof of Theorem~B}\label{sec:proofB}
Let $E\subseteq N_\R$ be the one-dimensional eigenspace of $\phi$ associated with $\mu_1$, choose $x\in M_\R$ such that $E^\perp=\{x=0\}$, and let $e$ be a generator of $E$, such that $x(e)>0$.

By techniques as in Sections~\ref{step3} and~\ref{step6} we can choose $v_1,\dots, v_m\in N$, such that $x(v_j)>0$ for $1\leq j\leq m$, $e$ lies in the interior of the cone $\sigma:=\sum_{j=1}^m\R_+v_j$, and $\sigma$ is invariant. 
Let 
$\Delta:=\{\sum_{j=1}^m\R_+\epsilon_jv_j\}_{\epsilon_j\in\{0,-1,+1\}^m}$.
Then $\Delta$ is a complete simplicial fan. The cones $\sigma$ and $\sum_{j=1}^m\R_+(-v_j)$ are invariant and all rays in $\Delta$ are mapped into one of these cones. Thus Corollary~\ref{invkoner} asserts that $f:X(\Delta)\dashrightarrow X(\Delta)$ is 1-stable. Also,  $\Delta$ admits a strictly convex $\Delta$-linear support function
of the form $\max_j|v_j^*|$, so $X(\Delta)$ is projective, see Section~\ref{divisorstycke}.  This completes the proof of Theorem~B.

We have the following partial analogue of Theorem~A'.
\begin{thmB'}
Let $f:(\C^*)^m\to(\C^*)^m$ be a monomial map. Suppose that the associated eigenvalues satisfy $|\mu_1|>|\mu_2|\geq|\mu_3|\geq\dots\geq|\mu_m|>0$. Then there exist a complete simplicial fan $\Delta'$ and $n_0\in\N$, such that $X(\Delta')$ is projective and $f^n:X(\Delta')\dashrightarrow X(\Delta')$ is $1$-stable for $n\geq n_0$.
\end{thmB'}

\begin{proof}
Let $E$, $e$, and $x$ be as in the proof of Theorem~B. Choose $v_1,\dots, v_m\in N$, such that $x(v_j)>0$ and $e\in\interior\sum_{j=1}^m\R_+v_j$, and construct a fan $\Delta$ as in the proof of Theorem~B. Then there is an $n_0\in\N$, such that, for $n\geq n_0$, $\sum_{j=1}^m\R_+v_j$ and $\sum_{j=1}^m\R_+(-v_j)$ are invariant under $\phi^n$; in particular $\phi^n$ maps all rays in $\Delta$ into one of these cones. Now Theorem~B' follows from Corollary~\ref{invkoner}.
\end{proof}
\begin{remark}\label{R301}
If we could find a regular refinement of $\Delta'$, not containing any rays in $E^\perp$, then we would get a smooth toric variety on which $f^n$ would be $1$-stable in Theorem~B'. However, when regularizing $\Delta'$ it seems difficult to control where the new rays appear; compare Section~\ref{sec:dimtwo}.
\end{remark}

By slightly modifying the proof of Theorem~A, we could solve the problem of making $f:X(\Delta)\dashrightarrow X(\Delta)$ $1$-stable in more general situations than the one in Theorem~A. Let us mention a result in the same vein as Theorem~B. 
\begin{prop}\label{specfall}
Let $\Delta$ be a (complete) simplicial fan in a lattice $N$, and let $f:X(\Delta)\dashrightarrow X(\Delta)$ be a monomial map. Assume that the associated eigenvalues satisfy $\mu_1 > \mu_2\geq \dots \geq \mu_m >0$. 

Let $E$ be the one-dimensional eigenspace of $\phi$ associated with $\mu_1$, and let $e$ be a generator of $E$. 
Assume that there are cones $\sigma^+,\sigma^-\in\Delta(m)$, such that $E^\perp\cap\sigma^\pm=\{0\}$ and $\pm e\in\interior \sigma^\pm$, and moreover that $E^\perp$ contains no rays of $\Delta$. Then there exists a simplicial refinement $\Delta'$ of $\Delta$ such that $f:X(\Delta')\dashrightarrow X(\Delta')$ is $1$-stable.
If $\Delta$ is projective, then $\Delta'$ can be chosen projective.
\end{prop}

\begin{proof}
Following Sections~\ref{step3} and~\ref{step6} we can find rational invariant simplicial cones $\tau^+$ and $\tau^-$ of dimension $m$, such that $\tau^\pm\subseteq \sigma^\pm$ and $\pm e\in\interior\tau^\pm$. By Lemma~\ref{barvinok} we can incorporate $\tau^\pm$ into $\Delta$ without adding extra rays. 

Since, by assumption, $E^\perp$ contains no rays of $\Delta$, all rays of $\Delta$ are eventually mapped into $\tau^+$ or $\tau^-$. Following Section~\ref{step7} we can incorporate the rays $\phi^n(\rho)$, where $\rho\in\Delta$ and $\phi^n(\rho)$ is not contained in $\tau^\pm$, into $\Delta$. More precisely, we can find a simplicial refinement $\Delta'$ of $\Delta$, 
such that $\tau^\pm\in\Delta'$ and each ray in $\Delta'$ is either mapped onto another ray in $\Delta'$ or into $\tau^+$ or $\tau^-$. Now $f:X(\Delta')\dashrightarrow X(\Delta')$ is $1$-stable by Corollary~\ref{invkoner}.

By Lemma~\ref{L401} $X(\Delta')$ in Proposition~\ref{specfall} can be chosen projective, provided that $\Delta$ is projective, cf. Section~\ref{sec:proofA}.
\end{proof}

Observe, in light of the above proof, that the way of constructing the fan $\Delta'$ in the proof of Theorem~A is in general far from being optimal in the sense that in general we refine $\Delta$ more than necessary. Indeed, if we would follow the strategy in Section~\ref{sec:proofA}, we would typically start out by adding rays inside the hyperplane $E^\perp$, see Section~\ref{step5}.

\section{Examples}\label{sec:examples}
We now illustrate our method for proving Theorem~A in dimensions~2 and~3. 
We also give examples illustrating the difficulties when the eigenvalues
have different signs.

Let $\mu$ be an eigenvalue of $\phi:N\to N$. Recall that either $\mu\in\Z$ or $\mu\notin\Q$. Suppose that $\mu$ is a simple eigenvalue, with corresponding one-dimensional eigenspace $E$. If $\mu\in\Z$, then $E$ and $E^\perp$ are rational. On the other hand if $\mu\notin\Q$, then $E$ is not rational.

\begin{ex}\label{dim2}
Let $N\cong \Z^2$ and let $\phi:N\to N$ be a $\Z$-linear map with eigenvectors $\mu_1>\mu_2>0$ and corresponding eigenspaces $E_1, E_2$. Then either $\mu_1,\mu_2\in\Z$ or $\mu_1,\mu_2\notin \Q$. In the first case, $E_1$ and $E_2$ are both rational and thus $T(\phi)$ and $\tred(\phi)$ are given by:
\begin{equation*}
  \xymatrix{
& (V_\emptyset,V_{12}) \POS []; [d]**\dir{-}, []; [dr]**\dir{-} & & 
& V_\emptyset \POS []; [d]**\dir{-}, []; [dr]**\dir{-} &
\\
& (V_1,V_{12})  \POS []; [dl]**\dir{-}, []; [d]**\dir{-} & (V_\emptyset,V_2)  \POS []; [d]**\dir{-}, []; [dr]**\dir{-}& \text{and} 
& V_1 \POS []; [d]**\dir{-} & V_2~~~~,
\\
(V_{12},V_{12}) & (V_1,V_1) & (V_2,V_{2}) & (V_\emptyset,V_\emptyset) 
& V_{12} &}
\end{equation*}
respectively. Here $V_\I=\sum_{i\in\I} E_i$ for $\I\subseteq \{1,2\}$; in particular $V_\emptyset = \{0\}$ and $V_{12}=N_\R$. 
In the second case, neither $E_1$ nor $E_2$ is rational and so the trees are given by: 
\begin{equation*}
  \xymatrix{
& (V_\emptyset,V_{12}) \POS []; [dl]**\dir{-}, []; [dr]**\dir{-} & 
&& 
V_\emptyset \POS []; [d]**\dir{-}
\\
(V_{12},V_{12}) && (V_\emptyset,V_\emptyset)
&& 
V_{12}}
\end{equation*}
In the first case the associated chambers are given by $C(V_\emptyset)=\{0\}$, 
$C(V_j)=V_j\setminus\{0\}$, 
and $C(V_{12})=N_\R\setminus (V_1\cup V_2)$, 
In the second case, $C(V_\emptyset)=\{0\}$ 
and $C(V_{12})=N_\R\setminus V_2$. 
Note that $N\cap C(V_{12})=N\setminus\{0\}$.
\end{ex}

\begin{ex}
  Let $N\cong \Z^3$ and let $\phi:N\to N$ be a $\Z$-linear map with eigenvalues $\mu_1 >\mu_2 >\mu_3>0$. Depending on whether or not the eigenvalues are rational, there are five possibilities of $\tred(\phi)$, of which three are the following
  \begin{equation}\label{letrad}
\xymatrix{
& V_\emptyset \POS []; [d]**\dir{-}, []; [dr]**\dir{-} & & 
&
V_\emptyset \POS []; [d]**\dir{-},  []; [dr]**\dir{-} & 
&
V_\emptyset \POS []; [d]**\dir{-}
\\
V_{12}\POS []; [d]**\dir{-}, []; [r]**\dir{-} & V_1 \POS []; [d]**\dir{-} & V_{23}  \POS []; [d]**\dir{-}, []; [dr]**\dir{-}
& 
&
V_{12} \POS []; [d]**\dir{-} & V_3 
& 
V_{123}
\\
V_{123} & V_{13} & V_2 & V_3
& 
V_{123} &
&
}
\end{equation}
Here we have used the notation from Example~\ref{dim2}. The first tree in \eqref{letrad} is obtained when all eigenvalues are integers, the second when $\mu_3$ is the unique integer eigenvalue, and the last tree when all eigenvalues are irrational. If $\mu_1$ or $\mu_2$ is the unique integer eigenvalue we get a tree of the same structure as the second tree, but with $V_{12}$ replaced by $V_{1}$ or $V_{13}$, respectively, and $V_3$ replaced by $V_{23}$ or $V_2$, respectively.
\end{ex}

If some of the eigenvalues of $\phi$ are negative, stabilization 
may not be possible, as the following example shows. 
\begin{ex}\label{negex}
  Let $N\cong \Z^3$ and assume that 
  $\phi:N\to N$ is a $\Z$-linear map with real, irrational eigenvalues
  satisfying $\mu_1>-\mu_2>-\mu_3>0$ and $\mu_1+\mu_2+\mu_3<0$.
  Then there is no simplicial fan $\Delta$ such that 
  $f:X(\Delta)\dashrightarrow X(\Delta)$ is 1-stable.

  Indeed, let $\Delta$ be any complete simplicial fan
  and let $\sigma_1\in\Delta$ be a cone containing a nonzero 
  eigenvector associated to the eigenvalue $\mu_1$ in its
  interior.
  It follows from Lemma~\ref{stablma} 
  that $\phi(\sigma_1)\subseteq\sigma_1$.
  Since the only nontrivial invariant rational
  subspaces of $N_\R$ are are $0$ and $N_\R$, $\sigma_1$ 
  must have dimension three. 
  By Lemma~\ref{spar}, this
  contradicts the assumption $\mu_1+\mu_2+\mu_3<0$.
\end{ex}
A concrete example is given by $\phi$ associated to 
the matrix 
$A=A_\phi=\left [ \begin{smallmatrix} 
    3&1&0\\ 1&-2&1\\0&1&-2 
  \end{smallmatrix} \right ]$. 
Then $\mu_1\approx 3.1997$, $\mu_2\approx -3.0855$, 
and $\mu_3\approx -1.1142$.

\end{document}